\documentclass[11pt]{amsart}

\usepackage{graphicx}              
\usepackage{amsfonts}              
\usepackage{float}				  
\usepackage{hyperref}
\usepackage{color}								
\usepackage{centernot}
\usepackage[integrals]{wasysym}
\usepackage[usenames,dvipsnames]{xcolor}
\usepackage{datetime}
\usepackage{epstopdf}

\usepackage{amsfonts,amssymb,amsmath,amsthm}
\usepackage[ps,all,arc,rotate]{xy}
\usepackage[lmargin=1in,rmargin=1in,tmargin=1in,bmargin=1in]{geometry}
\usepackage{fancyhdr}

\newtheorem{lemma}{Lemma}[section]
\newtheorem{theorem}{Theorem}[section]

\newtheorem{corollary}[lemma]{Corollary}

\newtheorem*{lemma*}{Lemma}
\theoremstyle{definition}
\newcommand{\R}{\mathbb{R}}
\newcommand{\C}{\mathbb{C}}
\newcommand{\N}{\mathbb{N}}

\newcommand{\real}{\operatorname{Re}}
\newcommand{\imag}{\operatorname{Im}}

\newtheorem{definition}{Definition}[section]
\newtheorem{remark}{Remark}[section]
\numberwithin{equation}{section}		 			% numbers equations within section
\numberwithin{figure}{section}			 			% numbers figure within section

\allowdisplaybreaks[4]

\begin{document}
\title[Explicit formulas of a generalized Ramanujan sum]{Explicit formulas of a generalized Ramanujan sum}
\author{Patrick K\"{u}hn}
\author{Nicolas Robles}
\address{Institut f\"{u}r Mathematik, Universit\"{a}t Z\"{u}rich, Winterthurerstrasse 190, CH-8057 Z\"{u}rich, Switzerland}
\email{patrick.kuehn@math.uzh.ch}
\email{nicolas.robles@math.uzh.ch}
%\date{This version is of \textcolor{red}{\today, at \currenttime.}}
\maketitle
\begin{abstract}
Explicit formulas involving a generalized Ramanujan sum are derived. An analogue of the prime number theorem is obtained and equivalences of the Riemann hypothesis are shown. Finally, explicit formulas of Bartz are generalized.
\end{abstract}

%\tableofcontents

\section{Introduction}
%
%For $\sigma, t \in \R$, we denote by $s = \sigma + it$ a general complex number. 
In \cite{ramanujan} Ramanujan introduced the following trigonometrical sum.
\begin{definition}
The Ramanujan sum is defined by
\begin{align} \label{ramanujansum}
{c_q}(n) = \sum_{\substack{(h,q) = 1}} {{e^{2\pi inh/q}}},
\end{align}
where $q$ and $n$ are in $\N$ and the summation is over a reduced residue system $\bmod \; q$. 
\end{definition}
Many properties were derived in \cite{ramanujan} and elaborated in \cite{hardy}. Cohen \cite{cohen} generalized this arithmetical function in the following way.
\begin{definition}
Let $\beta \in \N$. The $c_q^{(\beta)}(n)$ sum is defined by
\begin{align} \label{cohendef}
c_q^{(\beta )}(n) = \sum_{\substack{(h,{q^\beta })_\beta  = 1}} {{e^{2\pi inh/{q^\beta }}}},
\end{align}
where $h$ ranges over the non-negative integers less than $q^{\beta}$ such that $h$ and $q^{\beta}$ have no common $\beta$-th power divisors other than $1$.
\end{definition}
It follows immediately that when $\beta = 1$, \eqref{cohendef} becomes the Ramanujan sum \eqref{ramanujansum}. Among the most important properties of $c_q^{(\beta )}(n)$ we mention that it is a multiplicative function of $q$, i.e.
\[
c_{pq}^{(\beta )}(n) = c_p^{(\beta )}(n)c_q^{(\beta )}(n),\quad (p,q) = 1.
\]
%The moments of the average of generalized Ramanujan sums of this type have recently studied by the second author and Roy in \cite{rr01}.\\\\
%
%\indent By an explicit formula we mean, as is customary in the literature of the Riemann zeta-function, a formula relating a truncated sum of an arithmetical function on one side, and a sum over the zeros, $\rho$ and $-2k$ with $k=1,2,\cdots$, of the Riemann zeta-function on the other side. In this case, the arithmetical function of interest is $c_q^{(\beta)}(n)$ and the sum is on the interval $1 \le q \le x$.\\\\
%
\indent The purpose of this paper is to derive explicit formulas involving $c_q^{(\beta)}(n)$ in terms of the non-trivial zeros $\rho$ of the Riemann zeta-function and establish arithmetic theorems.
%
%The first explicit formula deals with the generalized Ramanujan sum.
%
\begin{definition}
Let $z \in \C$. The generalized divisor function $\sigma_z^{(\beta)}(n)$ is the sum of the $z^{\operatorname{th}}$ powers of those divisors of $n$ which are $\beta^{\operatorname{th}}$ powers of integers, i.e.
\[\sigma_z^{(\beta)}(n) = \sum_{{d^\beta }|n} {{d^{\beta z}}}. \]
\end{definition}
The object of study is the following.
\begin{definition}
For $x \ge 1$, we define
\[\mathfrak{C}^{(\beta )}(n,x) = \sum_{q \leqslant x} {c_q^{(\beta )}(n)}. \]
\end{definition}
For technical reasons we set
\begin{equation}
\mathfrak{C}^{\sharp,(\beta )}(n,x) = 
					\begin{cases}
					\mathfrak{C}^{(\beta )}(n,x), & \mbox{ if } x \notin \N,\\
					\mathfrak{C}^{(\beta )}(n,x) - \tfrac{1}{2}c_x^{(\beta)}(n), & \mbox{ if } x \in \N.
					\end{cases}
\end{equation}
The explicit formula for $\mathfrak{C}^{\sharp,(\beta )}(n,x)$ is then as follows.
\begin{theorem} \label{explicitcohenramanujan}
Let $\rho$ and $\rho_m$ denote non-trivial zeros of $\zeta(s)$ of multiplicity $1$ and $m \ge 2$ respectively. Fix integers $\beta$, $n$. There is an $1>\varepsilon > 0$ and a $T_0 = T_0(\varepsilon)$ such that \textnormal{(\ref{br-a})} and \textnormal{(\ref{br-b})} hold for a sequence $T_{\nu}$ and
\[
\mathfrak{C}^{\sharp,(\beta )}(n,x) =  - 2\sigma_1^{(\beta )}(n) + \sum_{\substack{|\gamma | < T_{\nu}}} {\frac{{\sigma_{1 - \rho /\beta }^{(\beta )}(n)}}{{\zeta '(\rho )}}\frac{{{x^\rho }}}{\rho }}  + {\rm K}_{T_{\nu}}(x) - \sum_{k = 1}^\infty  {\frac{{{{( - 1)}^{k}}(2\pi /{x})^{2k}}}{{(2k)!k\zeta (2k + 1)}}\sigma_{1 + 2k/\beta }^{(\beta )}(n)}  + E_{T_{\nu}}(x) ,
\]
where the error term satisfies
\[ 
E_{T_{\nu}}(x) \ll  \frac{x \log x}{T_{\nu}^{1-\varepsilon}} ,
\]
and where for the zeros of multiplicity $m \ge 2$ we have
\[
{\rm K}_{T_{\nu}}(x) = \sum_{m \geqslant 2} {\sum_{{|\gamma_m|<T_{\nu}}} {\kappa ({\rho _m},x)} } ,\quad \kappa ({\rho _m},x) = \frac{1}{{(m - 1)!}}\mathop {\lim }\limits_{s \to {\rho _m}} \frac{{{d^{m - 1}}}}{{d{s^{m - 1}}}}\bigg( {{{(s - {\rho _m})}^m}\frac{{\sigma_{1 - s/\beta }^{(\beta )}(n)}}{{\zeta (s)}}\frac{{{x^s}}}{s}} \bigg).
\]
Moreover, in the limit $\nu \to \infty$ we have
\[
 \mathfrak{C}^{\sharp,(\beta )}(n,x) =  - 2\sigma_1^{(\beta )}(n) + \lim_{\nu \to \infty} \sum_{\substack{|\gamma | < T_{\nu}}} {\frac{{\sigma_{1 - \rho /\beta }^{(\beta )}(n)}}{{\zeta '(\rho )}}\frac{{{x^\rho }}}{\rho }}  + \lim_{\nu \to \infty} {\rm K}_{T_{\nu}}(x) - \sum_{k = 1}^\infty  {\frac{{{{( - 1)}^{k}}(2\pi /{x})^{2k}}}{{(2k)!k\zeta (2k + 1)}}\sigma_{1 + 2k/\beta }^{(\beta )}(n)}.
\]
\end{theorem}
The next result is a generalization of a well-known theorem of Ramanujan which is of the same depth as the prime number theorem.
\begin{theorem} \label{line1theorem}
For fixed $\beta$ and $n$ in $\N$, we have 
\begin{align} \label{line1cohen}
\frac{{\sigma_{1 - s/\beta }^{(\beta )}(n)}}{{\zeta (s)}} = \sum_{q = 1}^\infty  {\frac{{c_q^{(\beta )}(n)}}{{{q^s}}}}
\end{align}
at all points on the line $\real(s)=1$.
\end{theorem}
\begin{corollary} \label{corollaryPNT1}
Let $\beta \in \N$. One has that
\begin{equation} \label{pnt_ramanujan}
\sum\limits_{q = 1}^\infty  {\frac{{c_q^{(\beta )}(n)}}{q}}  = 0, \quad \beta \ge 1,
\quad  \textnormal{and} \quad
\sum_{q = 1}^\infty  {\frac{{c_q^{(\beta )}(n)}}{{{q^\beta }}}}  = 
					\begin{cases}
					\tfrac{{\sigma_0^{(\beta )}(n)}}{{\zeta (\beta )}} & \mbox{ if } \beta > 1,\\
					0 & \mbox{ if } \beta =1.
					\end{cases}
\end{equation}
In particular 
\begin{align} \label{pnt_ramanujan2}
\sum_{q = 1}^\infty  {\frac{c_q(n)}{q}}  = 0 \quad \textnormal{and} \quad \sum_{q = 1}^\infty  {\frac{\mu(q)}{q}}  = 0.
\end{align}
\end{corollary}
It is possible to further extend the validity of \eqref{pnt_ramanujan} deeper into the critical strip, however, this is done at the cost of the Riemann hypothesis.
\begin{theorem} \label{equivalence1}
Let $\beta, n \in \N$. The Riemann hypothesis is true if and only if
\begin{align} \label{RH_equivalent}
\sum_{q = 1}^\infty  {\frac{{c_q^{(\beta )}(n)}}{{{q^s}}}}
\end{align}
is convergent and its sum is $\sigma_{1-s/\beta}^{(\beta)}(n) / \zeta(s)$, for every $s$ with $\sigma > \tfrac{1}{2}$.
\end{theorem}
This is a generalization of a theorem proved by Littlewood (see \cite{littlewood} and $\mathsection$14.25 of \cite{titchmarsh}) for the special case where $n =1$.
\begin{theorem} \label{equivalence2}
A necessary and sufficient condition for the Riemann hypothesis is
\begin{equation}
\mathfrak{C}^{(\beta)}(n,x) \ll_{n,\beta}  x^{\frac{1}{2} + \varepsilon}  \label{RH-1}
\end{equation}
for every $\varepsilon>0$.
\end{theorem}
We recall that the von Mangolt function $\Lambda(n)$ may be defined by
\[
\Lambda (n) = \sum_{d\delta  = n} {\mu (d)\log \delta}.
\]
Since $c_q^{(\beta)}(n)$ is a generalization of the M\"{o}bius function, we wish to construct a new $\Lambda (n)$ that incorporates the arithmetic information encoded in the variable $q$ and the parameter $\beta$.
\begin{definition}
For $\beta, k , m \in \N$ the generalized von Mangoldt function is defined as
\[ \Lambda _{k,m}^{(\beta )}(n) = \sum_{d\delta  = n} {c_d^{(\beta )}(m){{\log }^k}\delta }. \]
\end{definition}
We note the special case $\Lambda _{1,1}^{(1)}(n) = \Lambda(n)$. We will, for the sake of simplicity, work with $k=1$. The generalization for $k > 1$ requires dealing with results involving (computable) polynomials of degree $k-1$, see for instance $\mathsection$12.4 of \cite{ivic} as well as \cite{ivic1} and \cite{ivic2}.
\begin{definition}
The generalized Chebyshev function $\psi _m^{(\beta )}(x)$ and $\psi _{m}^{\sharp,(\beta )}(x)$ are defined by
\[
\psi _m^{(\beta )}(x) = \sum_{n \leqslant x} {\Lambda _{1,m}^{(\beta )}(n)}, \quad \operatorname{and}\quad \psi _{m}^{\sharp,(\beta )}(x) = \frac{1}{2}(\psi _m^{(\beta )}({x^ + }) + \psi _m^{(\beta )}({x^ - })).
\]
for $\beta, m \in \N$.
\end{definition}
%Let us set, as usual, 
%\[\psi _{m}^{\sharp,(\beta )}(x) = \frac{1}{2}(\psi _m^{(\beta )}({x^ + }) + \psi _m^{(\beta )}({x^ - })).\]
The explicit formula for the generalized Chebyshev function is given by the following result.
\begin{theorem} \label{explicitCohenChebyshev}
Let $c>1$, $\beta \in \N$, $x > m$, $T \ge 2$ and let $\left\langle x \right\rangle_{\beta}$ denote the distance from $x$ to the nearest interger such that $\Lambda_{1,m}^{\beta}(n)$ is not zero (other than $x$ itself). Then
\[
\psi _{m}^{\sharp,(\beta )}(x) = \sigma_{1 - 1/\beta }^{(\beta )}(m)x - \sum_{|\gamma | \leqslant T} {\sigma_{1 - \rho /\beta }^{(\beta )}(m)\frac{{{x^\rho }}}{\rho }}  - \sigma_1^{(\beta )}(m)\log (2\pi ) - \sum_{k = 1}^\infty  {\sigma_{1 + 2k/\beta }^{(\beta )}(m)\frac{{{x^{ - 2k}}}}{{2k}}}  + R(x,T),
\]
where
\[
R(x,T) \ll x^{\varepsilon}  \min \bigg( {1,\frac{x}{{T\left\langle x \right\rangle_{\beta} }}} \bigg) + \frac{x^{1+\varepsilon}\log x}{T} + \frac{{x{{\log }^2}T}}{T} ,
\]
for all $\varepsilon > 0$.
\end{theorem}
Taking into account the standard zero-free region of the Riemann-zeta function we obtain
\begin{theorem} \label{PNTzerofree}
One has that
\begin{align*}
  |\psi _m^{(\beta )}(x) - \sigma_{1 - 1/\beta }^{(\beta )}(m)x| \ll x^{1+\varepsilon} e^{- {c_2}{(\log x)^{1/2}}} ,  
\end{align*}
for $\beta \in \N$.
\end{theorem}
Moreover, on the Riemann hypothesis, one naturally obtains a better error term.
\begin{theorem} \label{PNTRH}
Assume RH. For $\beta \in \N$ one has that
\begin{align*}
\psi _m^{(\beta )}(x) = \sigma_{1 - 1/\beta }^{(\beta )}(m)x + O({x^{1/2+\varepsilon}} )
\end{align*}
for each $\varepsilon > 0$.
\end{theorem}
Our next set of results is concerned with a generalization of a function introduced by Bartz \cite{bartz1,bartz2}. The function introduced by Bartz was later used by Kaczorowski in \cite{kaczorowski} to study sums involving the M\"{o}bius function twisted by the cosine function. Let us set $\mathbb{H} = \{ x+iy, \; x \in \R, \; y >0 \}$.
\begin{definition}
Suppose that $z \in \mathbb{H}$, we define the $\varpi$ function by
\begin{align} \label{cohen_bartz_varpi}
\varpi _n^{(\beta )}(z) = \mathop {\lim }\limits_{m \to \infty } \sum_{\substack{\rho \\ 0 < \operatorname{Im} \rho  < {T_m}}} {\frac{{\sigma_{1 - \rho /\beta }^{(\beta )}(n)}}{{\zeta '(\rho )}}{e^{\rho z}}}. 
\end{align}
\end{definition}
The goal is to describe the analytic character of $\varpi _n^{(\beta )}(z)$. Specifically, we will construct its analytic continuation to a meromorphic function of $z$ on the whole complex plane and prove that it satisfies a functional equation. This functional equation takes into account values of $\varpi _n^{(\beta )}(z)$ at $z$ and at $\bar z$; therefore one may deduce the behavior of $\varpi _n^{(\beta )}(z)$ for $\imag(z)<0$. Finally, we will study the singularities and residues of $\varpi _n^{(\beta )}(z)$.
\begin{theorem} \label{bartz11}
The function $\varpi _n^{(\beta )}(z)$ is holomorphic on the upper half-plane $\mathbb{H}$ and for $z \in \mathbb{H}$ we have
\[2\pi i\varpi _n^{(\beta )}(z) = \varpi _{1,n}^{(\beta )}(z) + \varpi _{2,n}^{(\beta )}(z) - {e^{3z/2}}\sum_{q = 1}^\infty  {\frac{{c_q^{(\beta )}(n)}}{{{q^{3/2}}(z - \log q)}}} \]
where the last term on the right hand-side is a meromorphic function on the whole complex plane with poles at $z = \log q$ whenever $c_q^{(\beta)}(n)$ is not equal to zero. Moreover, 
\[
\varpi _{1,n}^{(\beta )}(z) = \int_{-1/2 + i \infty}^{-1/2} \frac{\sigma_{1-s/\beta}^{\beta}(n)}{\zeta(s)}e^{s z}ds
\]
is analytic on $\mathbb{H}$ and 
\[
\varpi _{2,n}^{(\beta )}(z) = \int_{-1/2}^{3/2} \frac{\sigma_{1-s/\beta}^{\beta}(n)}{\zeta(s)}e^{s z}ds
\] 
is regular on $\C$. 
\end{theorem}
\begin{remark}
This is done on the assumption that the non-trivial zeros are all simple. This is done for the sake of clarity, since straightforwad modifications are needed to relax this assumption. See $\mathsection$8 for further details.
\end{remark}
\begin{theorem} \label{bartz12}
The function $\varpi _n^{(\beta )}(z)$ can be continued analytically to a meromorphic function on $\C$ which satisfies the functional equation
\begin{align} \label{FE_cohen_bartz}
\varpi _n^{(\beta )}(z) + \overline{\varpi _n^{(\beta )}(\bar z)} =  A^{(\beta)}_n(z) = - \sum_{q = 1}^\infty  {\frac{{\mu (q)}}{q}\sum_{k = 0}^\infty  {\frac{1}{{k!}}\bigg\{ {{{\left( {{e^{ - z}}\frac{{2\pi i}}{q}} \right)}^k} + {{\left( { - {e^{ - z}}\frac{{2\pi i}}{q}} \right)}^k}} \bigg\}\sigma_{1 + k/\beta }^{(\beta )}(n)} },
\end{align}
where the function $A^{(\beta)}_n(z)$ is entire and satisfies
\[
A^{(\beta)}_n(z)=2\sum_{k = 1}^\infty  {\frac{{{{( - 1)}^k}{{(2\pi )}^{2k}}}}{{(2k)!}}\frac{{{e^{ - 2kz}}\sigma_{1 + k/\beta }^{(\beta )}(n)}}{{\zeta (1 + 2k)}}}.
\]
\end{theorem}
\begin{theorem} \label{bartz13}
The only singularities of $\varpi _n^{(\beta )}(z)$ are simple poles at the points $z = \log q$ on the real axis, where $q$ is an integer such that $c_q^{(\beta )}(n) \ne 0$, with residue
\[\mathop {\operatorname{res} }\limits_{z = \log n} \varpi _n^{(\beta )}(z) =  - \frac{1}{{2\pi i}}c_q^{(\beta )}(n).\]
\end{theorem}
\section{Proof of Theorem \ref{explicitcohenramanujan}}
%Originally, explicit formulae involving the Merten's function $M(x)=\sum_{n \le x}{\mu(n)}$ were proved by Titchmarsh under the assumption of the Riemann hypothesis, see $\mathsection$14.16 of \cite{titchmarsh}. Specifically, his result was conditioned on the following: for each $\varepsilon > 0$, $T \in \N$, there is a $t \in (T,T+1)$ such that 
%\[
% \frac{1}{\zeta(s)} \ll t^{\varepsilon},
%\]
%for $\frac{1}{2} \le \sigma \le 2$, if RH is true. Thus, under RH, for each $\varepsilon > 0$ we can find a sequence $T_n$ such that 
%\begin{equation} 
%n \le T_n \le n+1 , \ n = 1,2,3,\ldots  \label{br-aw}
%\end{equation}
%and
%\begin{equation} \label{br-bw}
% \left| {\zeta (\sigma  + i{T_n}) }\right|^{-1} < T_n^{{\varepsilon}}. 
%\end{equation}
%
In order to obtain unconditional results we use an idea put forward by Bartz \cite{bartz2}. The key is to use the following result of Montgomery, see \cite{montextreme} and Theorem 9.4 of \cite{ivic}. 
\begin{lemma}
 For any given $\varepsilon > 0$ there exists a real $T_0 = T_0(\varepsilon)$ such that for $T \ge T_0$ the following holds: between $T$ and $2T$ there exists a value of $t$ for which
 \[
  |\zeta(\sigma \pm it)|^{-1} < c_1 t^{\varepsilon} \textrm{ for } -1 \le \sigma \le 2,
 \]
with an absolute constant $c_1 > 0$.
\end{lemma}
That is, for each $\varepsilon > 0$, there is a sequence $T_{\nu}$, where
\begin{equation} 
 2^{\nu-1} T_0(\varepsilon) \le T_{\nu} \le 2^{\nu} T_0(\varepsilon), \ \nu = 1,2,3,\ldots \label{br-a}
\end{equation}
such that
\begin{equation} 
 |{\zeta(\sigma \pm iT_{\nu})}|^{-1} < c_1 T_{\nu}^{\varepsilon} \textrm{ for } -1 \le \sigma \le 2. \label{br-b}
\end{equation}
Finally, towards the end we will need the following bracketing condition: $T_m$ ($m \le T_m \le m+1$) are chosen so that
\begin{equation} 
\left| {{{\zeta (\sigma  + i{T_m})}}} \right|^{-1} < T_m^{{c_1}} \label{br-c}
\end{equation}
for $-1 \le \sigma \le 2$ and $c_1$ is an absolute constant. The existence of such a sequence of $T_m$ is guaranteed by Theorem 9.7 of \cite{titchmarsh}, which itself is a result of Valiron, \cite{valiron}.\\
We will use either bracketing (\ref{br-a})-(\ref{br-b}) or \eqref{br-c} depending on the necessity. These choices will lead to different bracketings of the sum over the zeros in the various explicit formulas appearing in the theorems of this note. 

The first immediate result is as follows.
\begin{lemma} \label{lemma}
The generalized divisor function $\sigma_z^{(\beta)}(n)$ satisfies the following bound for $z \in \C$, $n \in \N$
\[
  |\sigma_z^{(\beta)}(n)| \le \sigma_{\real(z)}^{(\beta)}(n) \le n^{\beta \max(0,\real(z))+1}.                                                          
\]
\end{lemma}
 In \cite{cohen} the following two properties of $c_q^{(\beta)}(n)$ are derived.
\begin{lemma} \label{lemma22}
For $\beta$ and $n$ integers one has
\begin{align} \label{cohenmoebius}
c_q^{(\beta )}(n) = \sum_{\substack{d|q \\ {d^\beta }|n}} {\mu \left( {\frac{q}{d}} \right){d^\beta }} \nonumber
\end{align}
where $\mu$ denotes the M\"{o}bius function.
\end{lemma}
\begin{lemma} \label{lemma23}
For $\real(s)>1$ and $\beta \in \N$ one has
\[\sum_{q = 1}^\infty  {\frac{{c_q^{(\beta )}(n)}}{{{q^{\beta s}}}}}  = \frac{{\sigma _{1 - s }^{(\beta )}(n)}}{{\zeta (\beta s)}}.\]
\end{lemma}
From Lemma \ref{lemma22} one has the following bound
\[
|c_q^{(\beta )}(n)| \leqslant \sum_{\substack{d|q \\ {d^\beta }|n}} {{d^\beta }}  \leqslant \sum_{{d^\beta }|n} {{d^\beta }}  = \sigma_1^{(\beta )}(n).
\]
Suppose $x$ is a fixed non-integer. Let us now consider the positively oriented path $\mathcal{C}$ made up of the line segments $[c-iT,c+iT,-2N-1+iT,-2N-1-iT]$ where $T$ is not the ordinate of a non-trivial zero. We set $a_q = c_q^{(\beta)}(n)$ and we use the lemma in $\mathsection$3.12 of \cite{titchmarsh} to see that we can take $\psi(q) = \sigma_1^{(\beta)}(n)$. We note that for $\sigma > 1$ we have
\[
\sum_{q = 1}^\infty  {\frac{{|c_q^{(\beta )}(n)|}}{{{q^\sigma }}}}  \leqslant \sigma_1^{(\beta)} (n)\sum_{q = 1}^\infty  {\frac{1}{{{q^\sigma }}}}  = \sigma_1^{(\beta)} (n)\zeta (\sigma ) \ll {\frac{1}{{\sigma  - 1}}} 
\]
so that $\alpha=1$. Moreover, if in that lemma we put $s=0$, $c=1+1/\log x$ and $w$ replaced by $s$, then we obtain
\[
\mathfrak{C}_0^{(\beta )}(n,x) = \frac{1}{{2\pi i}}\int_{c - iT}^{c + iT} {\frac{{\sigma_{1 - s/\beta }^{(\beta )}(n)}}{{\zeta (s)}}\frac{{{x^s}}}{s}ds} + E_{1,T}(x),
\]
where $E_{1,T}(x)$ is an error term that will be evaluated later.
If $x$ is an integer, then $\tfrac{1}{2}c_x^{(\beta)}(n)$ is to be substracted from the left-hand side. Then, by residue calculus we have
\[
\frac{1}{{2\pi i}}\oint\nolimits_\mathcal{C} {\frac{{\sigma_{1 - s/\beta }^{(\beta )}(n)}}{{\zeta (s)}}\frac{{{x^s}}}{s}ds}  = {R_0} + {R_\rho }(T) + {\rm K}(x,T) + {R_{ - 2k}}(N),
\]
where each term is given by the residues inside $\mathcal{C}$
\[
{R_0} = \mathop {\operatorname{res} }\limits_{s = 0} \frac{{\sigma_{1 - s/\beta }^{(\beta )}(n)}}{{\zeta (s)}}\frac{{{x^s}}}{s} =  - 2\sigma_1^{(\beta )}(n),
\]
and for $k = 1,2,3,\cdots$ we have
\[
{R_{ - 2k}}(N) = \sum_{k = 1}^N {\mathop {\operatorname{res} }\limits_{s =  - 2k} \frac{{\sigma_{1 - s/\beta }^{(\beta )}(n)}}{{\zeta (s)}}\frac{{{x^s}}}{s}}  = \sum_{k = 1}^N {\frac{{\sigma_{1 + 2k/\beta }^{(\beta )}(n)}}{{\zeta '( - 2k)}}\frac{{{x^{ - 2k}}}}{{ - 2k}}} = \sum_{k = 1}^N {\frac{{{{( - 1)}^{k - 1}}(2\pi /{x})^{2k}}}{{(2k)!k\zeta (2k + 1)}}\sigma_{1 + 2k/\beta }^{(\beta )}(n)}. 
\]
For the non-trivial zeros we must distinguish two cases. For the simple zeros $\rho$ we have
\[{R_\rho }(T) = \sum_{|\gamma | < T} {\mathop {\operatorname{res} }\limits_{s = \rho } \frac{{\sigma_{1 - s/\beta }^{(\beta )}(n)}}{{\zeta (s)}}\frac{{{x^s}}}{s}}  = \sum_{|\gamma | < T} {\frac{{\sigma_{1 - \rho /\beta }^{(\beta )}(n)}}{{\zeta '(\rho )}}\frac{{{x^\rho }}}{\rho }}, \]
and by the formula for the residues of order $m$ we see that ${\rm K}(x,T)$ is of the form indicated in the statement of the theorem. We now bound the vertical integral on the far left
\begin{align}
  \int_{ - 2N - 1 - iT}^{ - 2N - 1 + iT} {\frac{{\sigma_{1 - s/\beta }^{(\beta )}(n)}}{{\zeta (s)}}\frac{{{x^s}}}{s}ds} &= \int_{2N + 2 - iT}^{2N + 2 + iT} {\frac{{\sigma_{1 - (1 - s)/\beta }^{(\beta )}(n)}}{{\zeta (1 - s)}}\frac{{{x^{1 - s}}}}{{1 - s}}ds}  \nonumber \\
   &= \int_{2N + 2 - iT}^{2N + 2 + iT} {\frac{{{x^{1 - s}}}}{{1 - s}}\frac{{\sigma_{1 - (1 - s)/\beta }^{(\beta )}(n){2^{s - 1}}{\pi ^s}}}{{\cos (\tfrac{{\pi s}}{2})\Gamma (s)}}\frac{1}{{\zeta (s)}}ds}  \nonumber \\
   &\ll {\int_{ - T}^T {\frac{1}{T}{{\left( {\frac{{2\pi }}{x}} \right)}^{2N + 2}} {e^{(2N+3) \log(n) + 2N + 2 - (2N + \tfrac{3}{2})\log (2N + 2)}}dt} }, \nonumber
\end{align}
since by the use of Lemma \ref{lemma} we have $\sigma_{1 - (1 - s)/\beta}^{(\beta)}(n) \ll \sigma_{1 - (1 - 2 - 2N)/\beta }^{(\beta)}(n) \ll n^{2N + 3}$. 
This tends to zero as $N \to \infty$, for a fixed $T$ and a fixed $n$. Hence we are left with
\begin{align*}
\mathfrak{C}_0^{(\beta )}(n,x) & =  - 2\sigma_1^{(\beta )}(n) + \sum_{|\gamma | < T} {\frac{{\sigma_{1 - \rho /\beta }^{(\beta )}(n)}}{{\zeta '(\rho )}}\frac{{{x^\rho }}}{\rho }}  + {\rm K}(x,T) \\
&+ \sum_{k = 1}^\infty  {\frac{{{{( - 1)}^{k - 1}}(2\pi /{x})^{2k}}}{{(2k)!k\zeta (2k + 1)}}\sigma_{1 + 2k/\beta }^{(\beta )}(n)}  + \frac{1}{{2\pi i}}\bigg( {\int_{c - iT}^{ - \infty  - iT}   + \int_{ - \infty  + iT}^{c + iT}  } \bigg)\frac{{\sigma_{1 - s/\beta }^{(\beta )}(n)}}{{\zeta (s)}}\frac{{{x^s}}}{s}ds - E_{1,T}(x) \\
&  = - 2\sigma_1^{(\beta )}(n) + \sum_{|\gamma | < T} {\frac{{\sigma_{1 - \rho /\beta }^{(\beta )}(n)}}{{\zeta '(\rho )}}\frac{{{x^\rho }}}{\rho }}  + {\rm K}(x,T) + \sum_{k = 1}^\infty  {\frac{{{{( - 1)}^{k - 1}}(2\pi /{x})^{2k}}}{{(2k)!k\zeta (2k + 1)}}\sigma_{1 + 2k/\beta }^{(\beta )}(n)}  \\
&+ E_{2,T}(x) - E_{1,T}(x),
\end{align*}
where the last two terms are to be bounded. For the second integral, we split the range of integration in $(-\infty + iT, -1 +iT) \cup (-1+iT,c+iT)$ and we write
\begin{align}
  \int_{ - \infty  + iT}^{ - 1 + iT} {\frac{{\sigma_{1 - s/\beta }^{(\beta )}(n)}}{{\zeta (s)}}\frac{{{x^s}}}{s}ds} &= \int_{2 + iT}^{\infty  + iT} {\frac{{\sigma_{1 - (1 - s)/\beta }^{(\beta )}(n)}}{{\zeta (1 - s)}}\frac{{{x^{1 - s}}}}{{1 - s}}ds}  \nonumber \\
   &= \int_{2 + iT}^{\infty  + iT} {\frac{{{x^{1 - s}}}}{{1 - s}}\frac{{\sigma_{1 - (1 - s)/\beta }^{(\beta )}(n){2^{s - 1}}{\pi ^s}}}{{\cos (\tfrac{{\pi s}}{2})\Gamma (s)}}\frac{1}{{\zeta (s)}}ds}  \nonumber \\
   &\ll {\int_2^\infty  {\frac{1}{T}{{\left( {\frac{{2\pi }}{x}} \right)}^\sigma} {e^{((\beta + \sigma) \log n + \sigma  - (\sigma  - \tfrac{1}{2})\log \sigma }}d\sigma } }  \ll {\frac{1}{T x^2}} . \nonumber  
\end{align}
We can now choose for each $\varepsilon > 0$, $T = T_{\nu}$ satisfying (\ref{br-a}) and (\ref{br-b}) such that
\[
\frac{1}{{\zeta (s)}} \ll {t^\varepsilon }, \quad \frac{1}{2} \leqslant \sigma  \leqslant 2,\quad t = {T_\nu}.
\]
Thus the other part of the integral is
\[
\int_{ - 1 + i{T_{\nu}}}^{c + i{T_{\nu}}} {\frac{{\sigma_{1 - s/\beta }^{(\beta )}(n)}}{{\zeta (s)}}\frac{{{x^s}}}{s}ds}  \ll {\int_{ - 1}^{\min (\beta,c)} T_\nu ^{\varepsilon  - 1}{e^{(\beta - \sigma + 1) \log n}}{x^\sigma }d\sigma + \int_{\min (\beta,c)}^{c} T_\nu ^{\varepsilon  - 1}{x^\sigma }d\sigma } \ll x T_\nu ^{\varepsilon  - 1}.
\]
The integral over $(2-iT_{\nu},-\infty-iT_{\nu})$ is dealt with similarly. It remains to bound $E_{1,T_\nu}(x)$, i.e. the three error terms on the right-hand side of (3.12.1) in \cite{titchmarsh} . We have $\psi(q) = \sigma_1^{(\beta)}(n)$, $s = 0$, $c = 1 + \frac{1}{\log x}$ and $\alpha=1$. Inserting these yields
\begin{align*}
  E_{T_\nu}(x) & =  E_{2,T_\nu}(x) -  E_{1,T_\nu}(x)  \\
  &\ll x T_\nu ^{\varepsilon  - 1} +   \frac{x \log x}{T_\nu}  +  \frac{x \sigma_1^{\beta}(n) \log x}{T_\nu}  +  \frac{x \sigma_1^{\beta}(n)}{T_\nu} \ll  \frac{x \log x}{T_\nu^{1-\varepsilon}} . 
\end{align*}
If we assume that all non-trivial zeros are simple then term ${\rm K}_{T_{\nu}}(x)$ disappears. 
%It now follows from this theorem that
%\[
%\sum {\frac{{\sigma_{1 - \rho /\beta }^{(\beta )}(n)}}{{|\rho \zeta '(\rho )|}}} 
%\]
%is divergent. To see this, note that if it were convergent then
%\[
%\sum {\frac{{\sigma_{1 - \rho /\beta }^{(\beta )}(n)}}{{ \zeta '(\rho )}}\frac{{{x^\rho }}}{\rho }} 
%\]
%would be uniformly convergent over any finite interval, and this would mean that $\mathfrak{C}^{(\beta)}(n,x)$ is continuous, which it is not. 
Theorem \ref{explicitcohenramanujan} can be illustrated by plotting the explicit formula as follows.

\begin{figure}[H]
	\centering
	\includegraphics[scale=0.893]{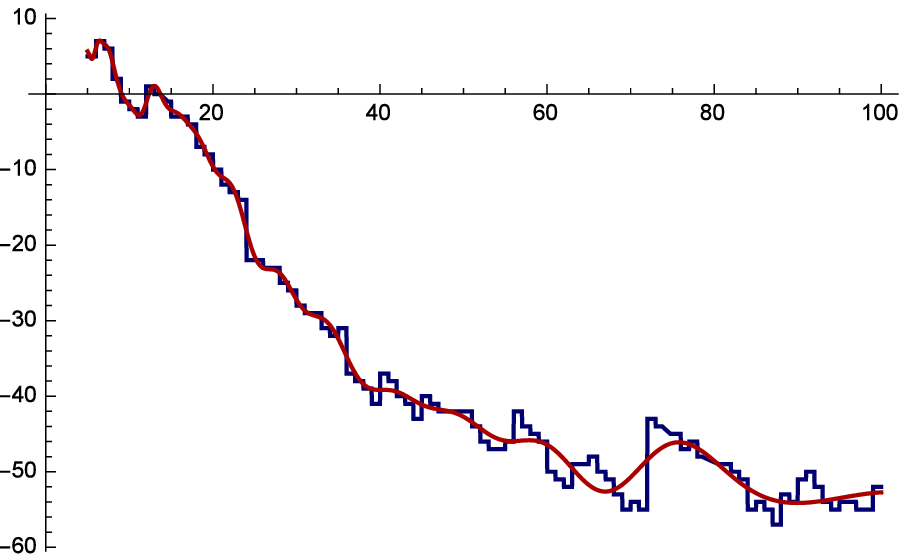}
	\includegraphics[scale=0.893]{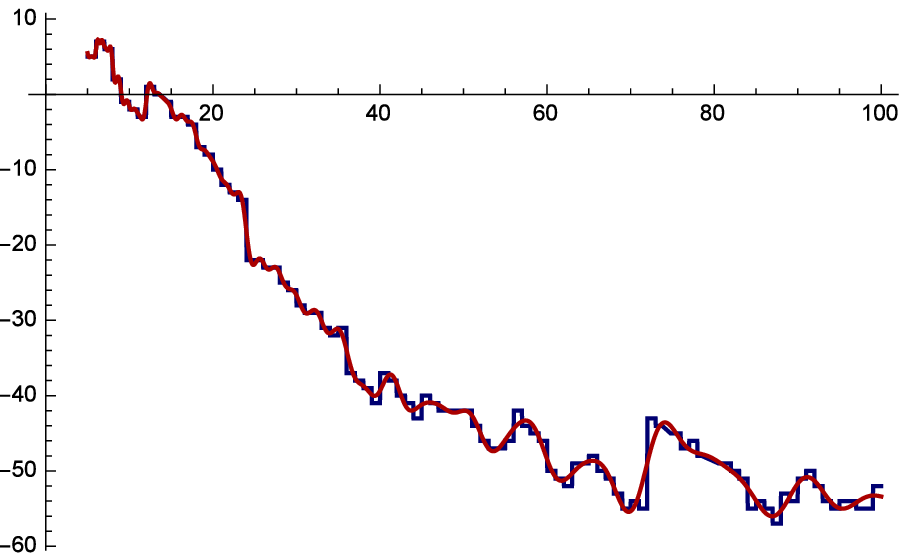}
	\caption{In blue: $\mathfrak{C}^{\sharp,(1)}(12,x)$, in red: the main terms of Theorem \ref{explicitcohenramanujan} with 5 and 25 pairs of zeros and $5 \le x \le 100$.}
\end{figure}

Increasing the value of $\beta$ does not affect the match. For $\beta=2$:

\begin{figure}[H]
	\centering
	\includegraphics[scale=0.893]{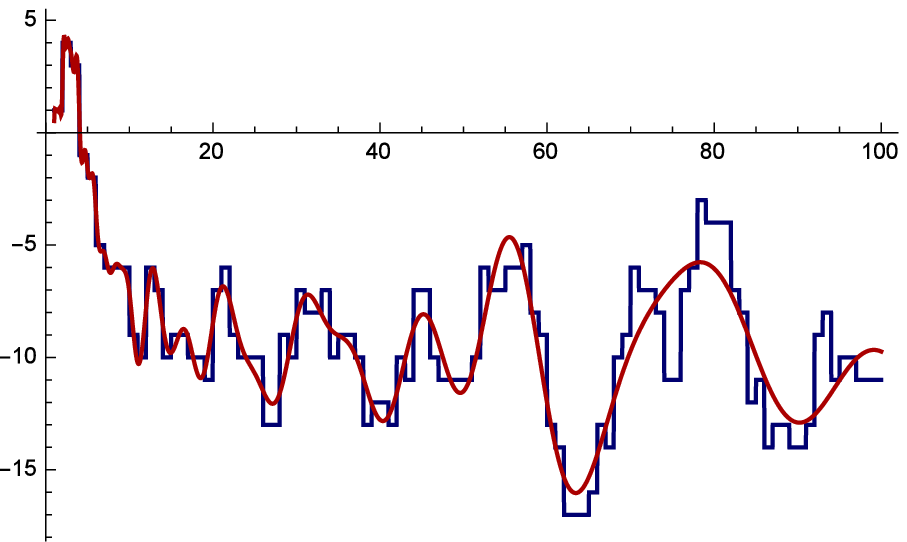}
	\includegraphics[scale=0.893]{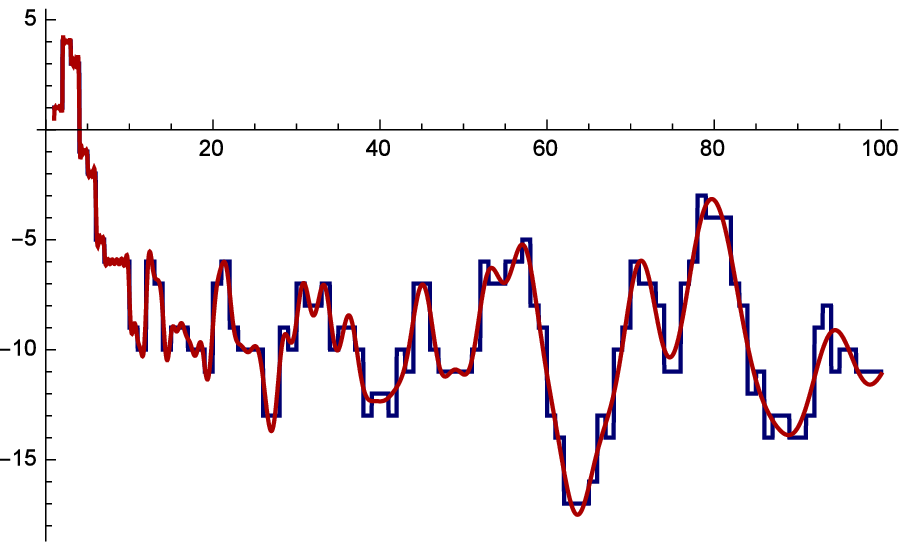}
	\caption{In blue: $\mathfrak{C}^{\sharp,(2)}(24,x)$, in red: the main terms of Theorem \ref{explicitcohenramanujan} with 5 and 25 pairs of zeros and $1 \le x \le 100$.}
\end{figure}

For $\beta =3$ we have the same effect

\begin{figure}[H]
	\centering
	\includegraphics[scale=0.893]{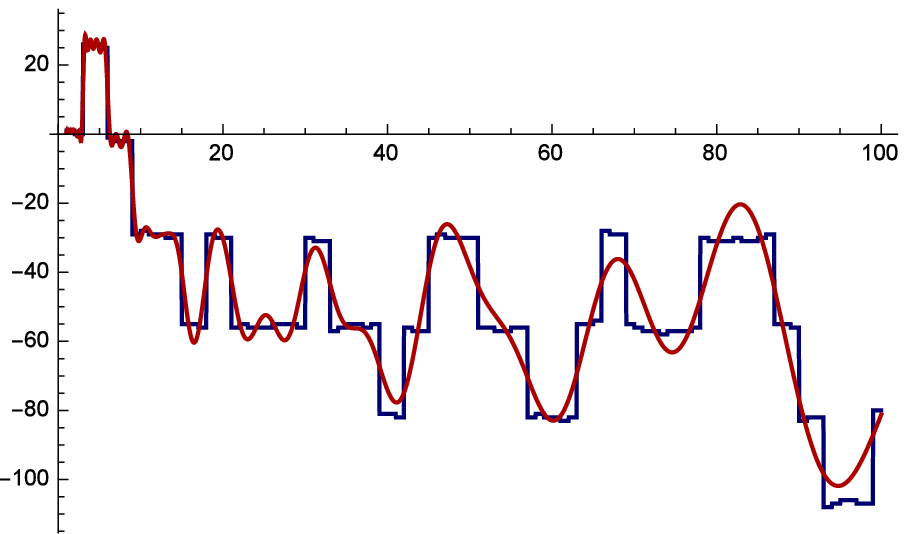}
	\includegraphics[scale=0.893]{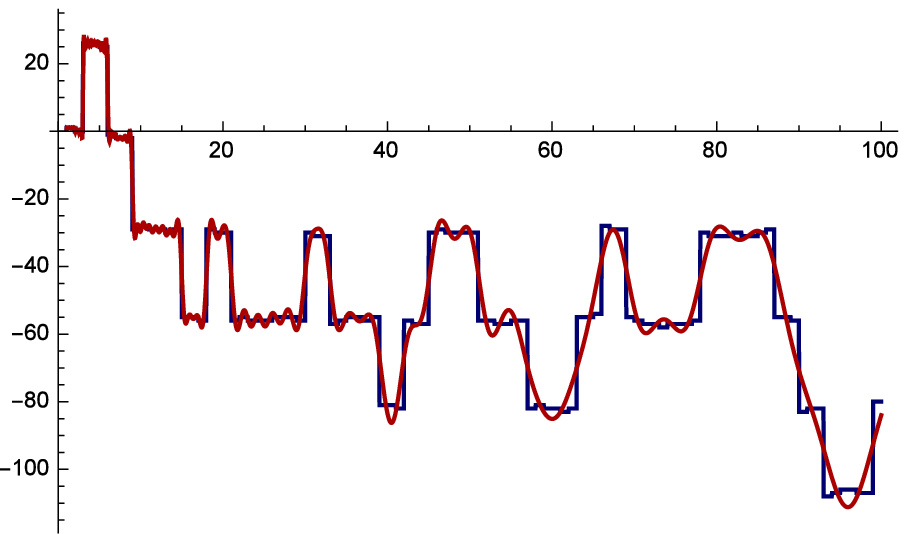}
	\caption{In blue: $\mathfrak{C}^{\sharp,(3)}(810,x)$, in red: the main terms of Theorem \ref{explicitcohenramanujan} with 5 and 25 pairs of zeros and $1 \le x \le 100$.}
\end{figure}
%
%%%%%%%%%%%%%%%%%%%%%%%%%%%%%%%%%%%%%%%%%%%%%%%%%%%%%%%%%%%%%%%%%%%%%%%%%%%%%%%%%%%%%%%%%%%%%%%%%%%%%%%%%%%%%%%%%%%%%%%
\section{Proof of Theorem \ref{line1theorem} and Corollary \ref{corollaryPNT1}}
We shall use the lemma in $\mathsection$3.12 of \cite{titchmarsh}. Take $a_q = c_q^{(\beta)}(n)$, $\alpha = 1$ and let $x$ be half an odd integer. Let $s = 1+it$, then
\begin{align*}
  \sum_{q < x} {\frac{{c_q^{(\beta )}(n)}}{{{q^s}}}} &= \frac{1}{{2\pi i}}\int_{c - iT}^{c + iT} {\frac{{\sigma_{1 - (s + w)/\beta }^{(\beta )}(n)}}{{\zeta (s + w)}}\frac{{{x^w}}}{w}dw}  + O\left( {\frac{{{x^c}}}{{Tc}}} \right) + O\left( {\frac{1}{T}\sigma_1^{(\beta )}(n)\log x} \right) \\
   &= \frac{{\sigma_{1 - s/\beta }^{(\beta )}(n)}}{{\zeta (s)}} + \frac{1}{{2\pi i}}\bigg( {\int_{c - iT}^{ - \delta  - iT} {}  + \int_{ - \delta  - iT}^{ - \delta  + iT} {}  + \int_{ - \delta  + iT}^{c + iT} {} } \bigg)\frac{{\sigma_{1 - (s + w)/\beta }^{(\beta )}(n)}}{{\zeta (s + w)}}\frac{{{x^w}}}{w}dw, 
\end{align*}
where $c>0$ and $\delta$ is small enough that $\zeta(s+w)$ has no zeros for
\[\operatorname{Re} (w) \geqslant  - \delta ,\quad |\operatorname{Im} (s + w)| = |t + \operatorname{Im} (w)| \leqslant |t| + T.\]
It is known from $\mathsection$3.6 of \cite{titchmarsh} that $\zeta(s)$ has no zeros in the region $\sigma > 1 - A \log^{-9}t$, where $A$ is a positive constant. 
Thus, we can take $\delta = A \log^{-9}T$. The contribution from the vertical integral is given by
\begin{align*}
  \int_{ - \delta  - iT}^{ - \delta  + iT} {\frac{{\sigma_{1 - (s + w)/\beta }^{(\beta )}(n)}}{{\zeta (s + w)}}\frac{{{x^w}}}{w}dw} &\ll {{x^{ - \delta }}n^{\delta} {{\log }^7}T \int_{ - T}^T {\frac{{{dv}}}{{\sqrt {{\delta ^2} + {v^2}} }}} } \ll {x^{ - \delta }}{n^{ \delta }}{\log ^8}T.   
\end{align*}
For the top horizontal integral we get
\begin{align*}
  \int_{ - \delta  + iT}^{c + iT} {\frac{{\sigma_{1 - (s + w)/\beta }^{(\beta )}(n)}}{{\zeta (s + w)}}\frac{{{x^w}}}{w}dw} &\ll  {\frac{{{{\log }^7}T}}{T} \bigg(\int_{ - \delta }^{\min(c,\beta - 1)} {{n^{ \beta - u }}{x^u}du} } + \int_{\min(c,\beta - 1)}^{c} x^u du \bigg)  \\
  &\ll {\frac{{{{\log }^7}T}}{T}{x^c}\bigg( \int_{ - \delta }^{\min(c,\beta - 1)} {{n^{ \beta - u}}du} } + c \bigg) \ll \frac{\log^7 T}{T} x^c n^{\delta}, 
\end{align*}
provided $x>1$. For the bottom horizontal integral we proceed the same way. Consequently, we have the following
\begin{align*}
  \sum_{q < x} {\frac{{c_q^{(\beta )}(n)}}{{{q^s}}}}  - \frac{{\sigma_{1 - s/\beta }^{(\beta )}(n)}}{{\zeta (s)}} \ll  {\frac{{{x^c}}}{{Tc}}} +  {\frac{1}{T}\sigma_1^{(\beta )}(n)\log x} + {x^{ - \delta }}{n^{ \delta }}{\log ^8}T +  {\frac{{{{\log }^7}T}}{T}{x^c}{n^{\delta }}} . 
\end{align*}
Now, we choose $c = 1/ \log x$ so that $x^c =e$. We take $T = \exp\{ (\log x)^{1/10}\}$ so that $\log T = (\log x)^{1/10}$, $\delta = A (\log x)^{-9/10}$ and $x^{\delta} = T^A$. Then it is seen that the right-hand side tends to zero as $x \to \infty$ and the result follows.

Corollary \ref{corollaryPNT1} follows by Lemma \ref{lemma23} for $\beta \ge 1$ and, by Theorem \eqref{line1theorem} with $\real(s) \ge 1$. If $s=1$ then the first equation follows. If in \eqref{line1cohen} we set $s=1$ then the second equation follows. Setting $s = \beta =1$ yields the third equation. Finally, putting $n=1$ in the third equation yields the fourth equation.

The plots of Corollary \ref{corollaryPNT1} are illustrated below.
\begin{figure}[H]
	\centering
	\includegraphics[scale=0.893]{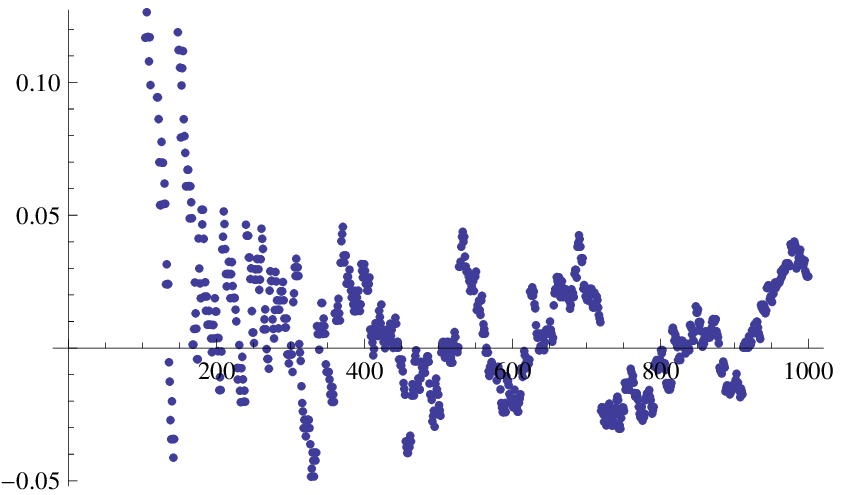}
	\includegraphics[scale=0.893]{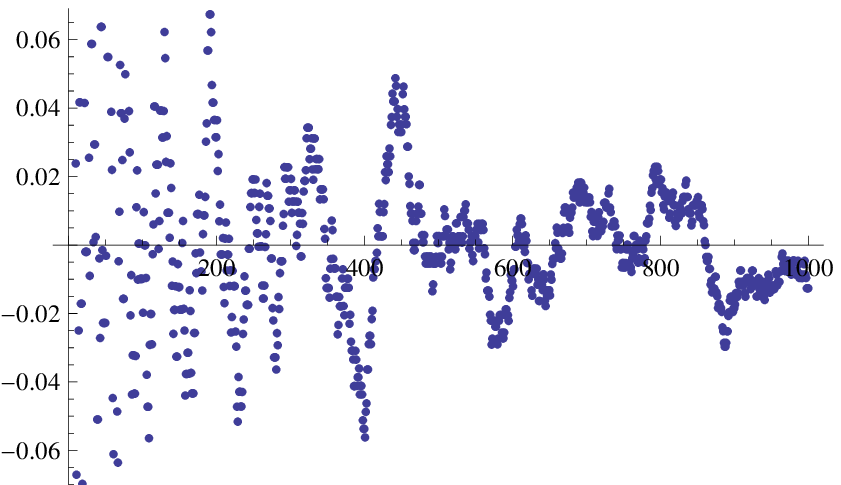}
	\caption{Plot of $\sum\nolimits_{q = 1}^x {c_q^{(1)}(24)/q}$ and of $\sum\nolimits_{q = 1}^x {c_q^{(2)}(24)/q}$ for $1 \le x \le 1000$.}
\end{figure}
\begin{figure}[H]
	\centering
	\includegraphics[scale=0.893]{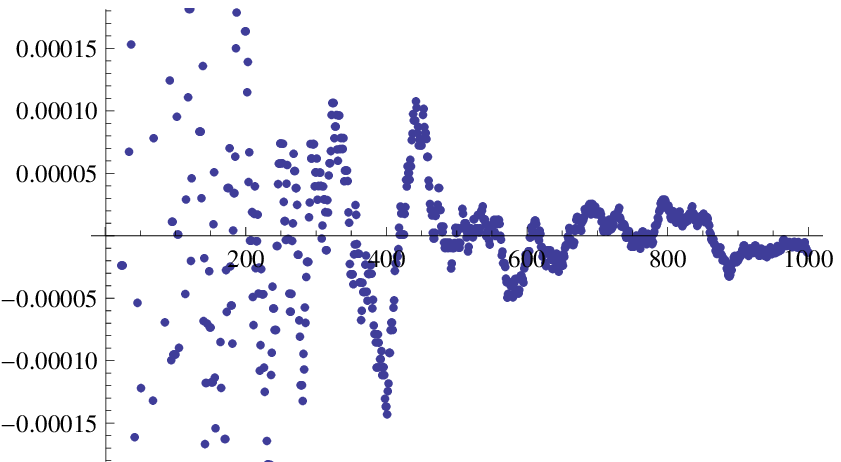}
	\includegraphics[scale=0.893]{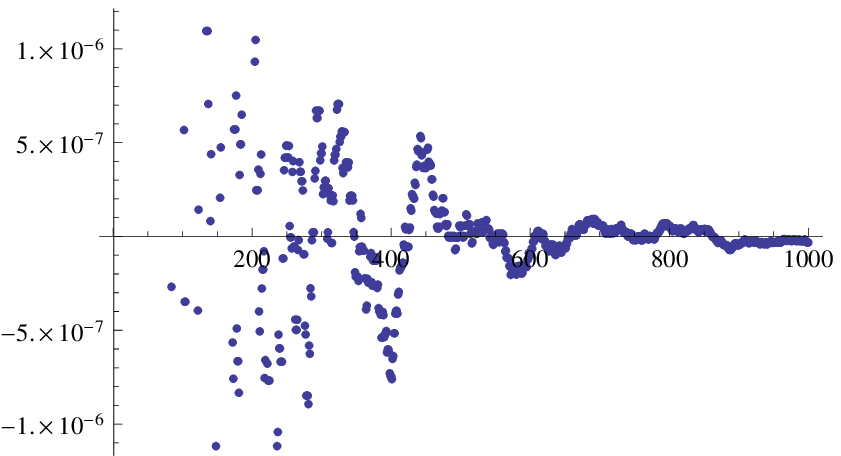}
	\caption{Plot of $\sum\nolimits_{q = 1}^x {c_q^{(2)}(24)/q^2} - \sigma_0^{(2)}(24)/\zeta (2)$ and of $\sum\nolimits_{q = 1}^x {c_q^{(3)}(24)/q^3} - \sigma_0^{(3)}(24)/\zeta (3)$ for $1 \le x \le 1000$.}
\end{figure}
%%%%%%%%%%%%%%%%%%%%%%%%%%%%%%%%%%%%%%%%%%%%%%%%%%%%%%%%%%%%%%%%%%%%%%%%%%%%%%%%%%%%%%%%%%%%%%%%%%%%%%%%%%%%%%%%%%%%%%%
\section{Proof of Theorem \ref{equivalence1}}
In the lemma of $\mathsection$3.12 of \cite{titchmarsh}, take $a_q = c_q^{(\beta)}(n)$, $f(s)=\sigma_{1-s/\beta}^{(\beta)}(n) / \zeta(s)$, $c=2$, $x$ half an odd integer. Then
\begin{align*}
  \sum_{q < x} {\frac{{c_q^{(\beta )}(n)}}{{{q^s}}}} &= \frac{1}{{2\pi i}}\int_{2 - iT}^{2 + iT} {\frac{{\sigma_{1 - (s + w)/\beta }^{(\beta )}(n)}}{{\zeta (s + w)}}\frac{{{x^w}}}{w}dw}  + O\left( {\frac{{{x^2}}}{T}} \right)  \\
   &= \frac{1}{{2\pi i}}\bigg( {\int_{2 - iT}^{\tfrac{1}{2} - \sigma  + \delta  - iT} {}  + \int_{\tfrac{1}{2} - \sigma  + \delta  - iT}^{\tfrac{1}{2} - \sigma  + \delta  + iT} {}  + \int_{\tfrac{1}{2} - \sigma  + \delta  + iT}^{2 + iT} {} } \bigg)\frac{{\sigma_{1 - (s + w)/\beta }^{(\beta )}(n)}}{{\zeta (s + w)}}\frac{{{x^w}}}{w}dw  \\
   &+ \frac{{\sigma_{1 - s/\beta }^{(\beta )}(n)}}{{\zeta (s)}} + O\left( {\frac{{{x^2}}}{T}} \right), 
\end{align*}
where $0 < \delta < \sigma - \tfrac{1}{2}$. If we assume RH, then $\zeta(s) \ll t^{\varepsilon}$ for $\sigma \ge \tfrac{1}{2}$ and $\forall \varepsilon >0$ so that the first and third integrals are 
\[
  \ll {{T^{ - 1 + \varepsilon }} \bigg( \int_{\tfrac{1}{2} - \sigma  + \delta }^{\min(\beta - \sigma,2)} {{n^{\beta - \sigma + v}}{x^v}dv} } + \int_{\min(\beta - \sigma,2)}^{2} x^v dv \bigg) \ll {T^{ - 1 + \varepsilon }}{x^2},
\]
provided $x > 1$. The second integral is
\[
 \ll {{x^{\tfrac{1}{2} - \sigma  + \delta }} n^{\beta + \delta + 1} \int_{ - T}^T {{{(1 + |t|)}^{ - 1 + \varepsilon }}dt} } \ll {{x^{\tfrac{1}{2} - \sigma  + \delta }}{T^\varepsilon }} . 
\]
Thus we have
\[
\sum_{q < x} {\frac{{c_q^{(\beta )}(n)}}{{{q^s}}}}  = \frac{{\sigma_{1 - s/\beta }^{(\beta )}(n)}}{{\zeta (s)}} + O({x^{\tfrac{1}{2} - \sigma  + \delta }}{T^\varepsilon }) + O({x^2}{T^{\varepsilon  - 1}}).
\]
Taking $T=x^3$ the $O$-terms tend to zero as $x \to \infty$, and the result follows. Conversely, if \eqref{RH_equivalent} is convergent for $\sigma > \tfrac{1}{2}$, then it is uniformly convergent for $\sigma \ge \sigma_ 0 > \tfrac{1}{2}$, and so in this region it represents an analytic function, which is $\sigma_{1-s/\beta}^{(\beta)}(n) / \zeta(s)$ for $\sigma > 1$ and so throughout the region. This means that the Riemann hypothesis is true and the proof is now complete.
%%%%%%%%%%%%%%%%%%%%%%%%%%%%%%%%%%%%%%%%%%%%%%%%%%%%%%%%%%%%%%%%%%%%%%%%%%%%%%%%%%%%%%%%%%%%%%%%%%%%%%%%%%%%%%%%%%%%%%%
\section{Proof of Theorem \ref{equivalence2}}
In the lemma of $\mathsection$3.12 of \cite{titchmarsh}, take $a_q = c_q^{(\beta)}(n)$, $f(w)=\sigma_{1-w/\beta}^{(\beta)}(n) / \zeta(w)$, $c=2$, $s = 0$, $\delta > 0$ and $x$ half an odd integer. Then
\begin{align*}
\mathfrak{C}^{(\beta)}(n,x) & = \sum_{q < x} c_q^{(\beta)}(n) = \frac{1}{2 \pi i} \int_{2 - iT}^{2 + iT} {\frac{{\sigma_{1 - w/\beta }^{(\beta )}(n)}}{{\zeta (w)}}\frac{{{x^w}}}{w}dw}  + O\bigg( {\frac{{{x^2 \sigma_1^{(\beta)}(n)}}}{T}} \bigg) \\
& = \frac{1}{{2\pi i}}\bigg( {\int_{2 - iT}^{\tfrac{1}{2} + \delta  - iT} {}  + \int_{\tfrac{1}{2}  + \delta  - iT}^{\tfrac{1}{2} + \delta  + iT} {}  + \int_{\tfrac{1}{2}  + \delta  + iT}^{2 + iT} {} } \bigg) {\frac{{\sigma_{1 - w/\beta }^{(\beta )}(n)}}{{\zeta (w)}}\frac{{{x^w}}}{w}dw}  + O\left( {\frac{{{x^2}}}{T}} \right) \\
&\ll  \int_{-T}^{T} x^{1/2 + \delta} \sigma_{1 - (\frac{1}{2} + \delta)/\beta}^{(\beta )}(n) (1 + |t|)^{\varepsilon - 1} dt +  T^{\varepsilon - 1} x^2 \sigma_{1 - (\frac{1}{2}+ \delta)/\beta}^{(\beta )}(n) +  {\frac{{{x^2 }}}{T}}  \\
&\ll T^{\varepsilon} x^{1/2 + \delta}  +  T^{\varepsilon - 1} x^2  + {\frac{{{x^2 }}}{T}} .
\end{align*}
If we take $T = x^2$, then 
\[
 \mathfrak{C}^{(\beta)}(n,x) \ll x^{1/2 + \varepsilon'} .
\]
for $\varepsilon' = 2 \varepsilon + \delta > 0$ and (\ref{RH-1}) follows. Conversely, if (\ref{RH-1}) holds, then by Abel summation 
\begin{align*}
 \sum_{q \le x} \frac{c_q^{(\beta)}(n)}{q^s} \ll \mathfrak{C}^{(\beta)}(n,x) x^{-\sigma} t +  \int_1^x \mathfrak{C}^{(\beta)}(n,t) t^{-\sigma - 1} dt \ll  x^{\frac{1}{2} - \sigma + \varepsilon}  
\end{align*}
converges as $x \to \infty$ for  $\sigma > \frac{1}{2}$, and thus the Riemann hypothesis follows.
%%%%%%%%%%%%%%%%%%%%%%%%%%%%%%%%%%%%%%%%%%%%%%%%%%%%%%%%%%%%%%%%%%%%%%%%%%%%%%%%%%%%%%%%%%%%%%%%%%%%%%%%%%%%%%%%%%%%%%%
\section{Proof of Theorem \ref{explicitCohenChebyshev}}
First, the Dirichlet series are given by the following result.
\begin{lemma} \label{dirichletcohenmangoldt}
For $\real(s)>1$ and $\beta, k \in \N$ one has
\[\sum_{n = 1}^\infty  {\frac{{ \Lambda _{k,m}^{(\beta )}(n)}}{{{n^s}}}}  = {( - 1)^k}\sigma_{1 - s/\beta }^{(\beta )}(m)\frac{{{\zeta ^{(k)}}(s)}}{{\zeta (s)}},\]
where $\zeta^{(k)}(s)$ is the $k^{\operatorname{th}}$ derivative of the Riemann zeta-function.
\end{lemma}
\begin{proof}
By Lemma \ref{lemma23}, and 
\[
  \sum_{n=1}^{\infty} \frac{\log^k n}{n^s} = (-1)^k \zeta^{(k)}(s)
\]
for $\real(s) >1$, the result follows by Dirichlet convolution.
\end{proof}
From Lemma \ref{dirichletcohenmangoldt} we deduce that 
\[
 \Lambda _{1,m}^{(\beta )}(n) \ll n^{\varepsilon} 
\]
for each $\varepsilon > 0$, otherwise the sum would not be absolutely convergent for $\real(s) > 1$.
It is known, see for instance Lemma 12.2 of \cite{montvau}, that for each real number $T \ge 2$ there is a $T_1$, $T \le T_1 \le T+1$, such that
\[
\frac{{\zeta '}}{\zeta }(\sigma  + i{T_1}) \ll {(\log T)^2}
\]
uniformly for $-1 \le \sigma \le 2$. By using Perron's inversion formula with $\sigma_0 = 1 + 1/ \log x$ we obtain
\[
\psi _{0,m}^{(\beta )}(x) =  - \frac{1}{{2\pi i}}\int_{{\sigma _0} - i{T_1}}^{{\sigma _0} + i{T_1}} {\sigma_{1 - s/\beta }^{(\beta )}(m)\frac{{\zeta '}}{\zeta }(s)\frac{{{x^s}}}{s}ds}  + {R_1},
\]
where
\[
{R_1} \ll \sum_{\substack{x/2 < n < 2x \\ n \ne x}} {\Lambda_{1,m}^{(\beta )}(n)\min \left( {1,\frac{x}{{T|x - n|}}} \right) + \frac{x}{T}} \sum_{n = 1}^\infty  {\frac{{\Lambda_{1,m}^{(\beta )}(n)}}{{{n^{{\sigma _0}}}}}} . 
\]
The second sum is
\[ 
- \sigma_{1 - {\sigma _0}/\beta }^{(\beta )}(m)\frac{{\zeta '}}{\zeta }({\sigma _0}) \asymp \frac{{\sigma_{1 - {\sigma _0}/\beta }^{(\beta )}(m)}}{{{\sigma _0} - 1}} = \sigma_{1 - (1 + 1/\log x)/\beta }^{(\beta )}(m)\log x.
\]
The term involving the generalized divisor function can be bounded in the following way:
\[
 \sigma_{1 - (1 + 1/\log x)/\beta }^{(\beta )}(m) \le m^{\beta + 1/\log x}
\]
if $\frac{1}{\log x} \le \beta - 1$, and $\le m$ otherwise. In both cases, this is bounded in $x$. For the first sum we do as follows. The terms for which $x + 1 \le n < 2x$ contribute an amount which is
\[ 
\ll \sum_{x + 1 \leqslant n < 2x} {\frac{{ x^{1+\varepsilon}}}{{T(n - x)}}} \ll  \frac{x^{1+\varepsilon}\log x}{T} .
\]
The terms for which $x / 2 < n \le x-1$ are dealt with in a similar way. The remaining terms for which $x -1 < n < x+1$ contribute an amount which is
\[ 
\ll x^{\varepsilon} \min \bigg( {1,\frac{x}{{T\left\langle x \right\rangle_{\beta} }}} \bigg) ,
\]
therefore, the final bound for $R_1$ is
\[
{R_1} \ll x^{\varepsilon}  \min \bigg( {1,\frac{x}{{T\left\langle x \right\rangle_{\beta} }}} \bigg) + \frac{x^{1+\varepsilon}\log x}{T} .
\]
We denote by $N$ an odd positive integer and by $\mathcal{D}$ the contour consisting of line segments connecting $\sigma_0 - iT_1 , - N - iT_1, -N +iT_1, \sigma_0 + iT_1$. An application of Cauchy's residue theorem yields
\[
\psi _{0,m}^{(\beta )}(x) = {M_0} + {M_1} + {M_\rho } + {M_{ - 2k}} + {R_1} + {R_2}
\]
where the terms on the right-hand sides are the residues at $s=0$, $s=1$, the non-trivial zeros $\rho$ and at the trivial zeros $-2k$ for $k=1,2,3,\cdots$, respectively, and where
\[
{R_2} =  - \frac{1}{{2\pi i}}\oint\nolimits_\mathcal{D} {\sigma_{1 - s/\beta }^{(\beta )}(m)\frac{{\zeta '}}{\zeta }(s)\frac{{{x^s}}}{s}ds}. 
\]
For the constant term we have
\[
{M_0} = \mathop {\operatorname{res} }\limits_{s = 0} \sigma_{1 - s/\beta }^{(\beta )}(m)\frac{{\zeta '}}{\zeta }(s)\frac{{{x^s}}}{s} = \sigma_1^{(\beta )}(m)\frac{{\zeta '}}{\zeta }(0) = \sigma_1^{(\beta )}(m)\log (2\pi ),
\]
and for the leading term
\[
{M_1} = \mathop {\operatorname{res} }\limits_{s = 1} \sigma_{1 - s/\beta }^{(\beta )}(m)\frac{{\zeta '}}{\zeta }(s)\frac{{{x^s}}}{s} = \sigma_{1 - 1/\beta }^{(\beta )}(m)x.
\]
The fluctuaring term coming from the non-trivial zeros yields
\[
{M_\rho } = \sum_\rho  {\mathop {\operatorname{res} }\limits_{s = \rho } \sigma_{1 - s/\beta }^{(\beta )}(m)\frac{{\zeta '}}{\zeta }(s)\frac{{{x^s}}}{s}}  = \sum_\rho  {\sigma_{1 - \rho /\beta }^{(\beta )}(m)\frac{{{x^\rho }}}{\rho }}, 
\]
by the use of the logarithmic derivative of the Hadamdard product of the Riemann zeta-function, and finally for the trivial zeros
\[
{M_{ - 2k}} = \sum_{k = 1}^\infty  {\mathop {\operatorname{res} }\limits_{s =  - 2k} \sigma_{1 - s/\beta }^{(\beta )}(m)\frac{{\zeta '}}{\zeta }(s)\frac{{{x^s}}}{s}}  = \sum_{k = 1}^\infty  {\sigma_{1 + 2k/\beta }^{(\beta )}(m)\frac{{{x^{ - 2k}}}}{{ - 2k}}}. 
\]
Since $|\sigma  \pm i{T_1}| \geqslant T$, we see, by our choice of $T_1$, that
\begin{align*}
  \int_{ - 1 \pm i{T_1}}^{{\sigma _0} \pm i{T_1}} {\sigma_{1 - s/\beta }^{(\beta )}(m)\frac{{\zeta '}}{\zeta }(s)\frac{{{x^s}}}{s}ds}  &\ll \frac{{{{\log }^2}T}}{T} \bigg(\int_{ - 1}^{{\min(\beta,\sigma_0)}} \left(\frac{x}{m}\right)^{\sigma} d\sigma  + \int_{\min(\beta,\sigma_0)}^{{\sigma _0}} x^{\sigma} d\sigma \bigg) \\
   &\ll \frac{{x{{\log }^2}T}}{{T\log x}} \ll  \frac{{x{{\log }^2}T}}{T} . 
\end{align*}
Next, we invoke the following result, see Lemma 12.4 of \cite{montvau}: if $\mathcal{A}$ denotes the set of points $s \in \mathbb{C}$ such that $\sigma \le -1$ and $|s+2k| \ge 1/4$ for every positive integer $k$, then
\[
\frac{{\zeta '}}{\zeta }(s) \ll \log (|s| + 1) 
\]
uniformly for $s \in \mathcal{A}$. This, combined with the fact that 
\[
\frac{\log |\sigma  \pm i{T_1}|}{|\sigma  \pm i{T_1}|} \ll \frac{\log T}{T} ,
\]
gives us
\begin{align*}
  \int_{ - N \pm i{T_1}}^{ - 1 \pm i{T_1}} {\sigma_{1 - s/\beta }^{(\beta )}(m)\frac{{\zeta '}}{\zeta }(s)\frac{{{x^s}}}{s}ds}  &\ll \frac{{\log T}}{T}\int_{ - \infty }^{ - 1} {{\left(\frac{x}{m}\right)^\sigma }d\sigma } \ll \frac{{\log T}}{{xT\log x}} \ll  \frac{{\log T}}{T} .
\end{align*}
Thus this bounds the horizontal integrals. Finally, for the left vertical integral, we have that $|-N +iT| \ge N$ and by the above result regarding the bound of the logarithmic derivative we also see that
\begin{align*}
  \int_{ - N - i{T_1}}^{ - N + i{T_1}} {\sigma_{1 - s/\beta }^{(\beta )}(m)\frac{{\zeta '}}{\zeta }(s)\frac{{{x^s}}}{s}ds}  &\ll \frac{{\log NT}}{N}{x^{ - N}}\sigma_{1 + N/\beta }^{(\beta )}(m)\int_{ - {T_1}}^{{T_1}} {dt} \ll  \frac{{T\log NT}}{{N}} \left(\frac{m}{x}\right)^N . 
\end{align*}
This last term goes to $0$ as $N \to \infty$ since $x > m$.
%%%%%%%%%%%%%%%%%%%%%%%%%%%%%%%%%%%%%%%%%%%%%%%%%%%%%%%%%%%%%%%%%%%%%%%%%%%%%%%%%%%%%%%%%%%%%%%%%%%%%%%%%%%%%%%%%%%%%%%
\section{Proof of Theorems \ref{PNTzerofree} and \ref{PNTRH}}
Let us denote by $\rho = \beta^* + i\gamma$ a non-trivial zero. For this, we will use the result that if $|\gamma| < T$, where $T$ is large, then $\beta^* < 1 - c_1 / \log T$, where $c_1$ is a positive absolute constant. This immediately yields
\[
|{x^\rho }| = {x^{{\beta ^*}}} < x e^{- {c_1}\log x/\log T}.
\]
Moreover, $|\rho| \ge \gamma$, for $\gamma >0$. We recall that the number of zeros $N(t)$ up to height $t$ is (Chapter 18 of \cite{davenport}.)
\[
N(t) = \frac{t}{{2\pi }}\log \frac{t}{{2\pi }} - \frac{t}{{2\pi }} + O(\log t) \ll  t \log t .
\]
We need to estimate the following sum
\[
\sum_{0 < \gamma  < T} {\frac{{\sigma_{1 - \gamma /\beta }^{(\beta )}(m)}}{\gamma }} = \sum_{1 < \gamma  < T} {\frac{{\sigma_{1 - \gamma /\beta }^{(\beta )}(m)}}{\gamma }}. 
\]
This is
\[
\ll \int_1^T {\frac{{\sigma_{1 - t/\beta }^{(\beta )}(m)}}{{{t^2}}}N(t)dt} \ll  m \int_1^{\beta} {\frac{{\log t}}{t}dt} + m^{\beta + 1}\int_\beta^{T} {\frac{{ \log t}}{t m^t}dt} \ll {\log ^2}T .
\]
Therefore,
\[
\sum_{|\gamma | < T} {\left| {\sigma_{1 - \rho /\beta }^{(\beta )}(m)\frac{{{x^\rho }}}{\rho }} \right|}  \ll  x{(\log T)^2} e^{- {c_1}\log x/\log T} .
\]
Without loss of generality we take $x$ to be an integer in which case the error term of the explicit formula of Theorem \ref{explicitCohenChebyshev} becomes
\[
R(x,T) \ll \frac{x^{1+\varepsilon}\log x}{T} + \frac{{x{{\log }^2}T}}{T} .
\]
Finally, we can bound the sum
\[
 \sum_{k = 1}^{\infty} \sigma_{1 + 2k/\beta}^{(\beta)} (m) \frac{x^{-2k}}{2k} \le m^{\beta + 1} \sum_{k = 1}^{\infty} m^{2k} \frac{x^{-2k}}{2k} = \frac{1}{2} m^{\beta + 1} \log \bigg(1 - \bigg(\frac{x}{m}\bigg)^{-2} \bigg) = o(1).
\]
Thus, we have the following
\[
|\psi _m^{(\beta )}(x) - \sigma_{1 - 1/\beta }^{(\beta )}(m)x| \ll  \frac{x^{1+\varepsilon}\log x}{T} + \frac{{x{{\log }^2}T}}{T} + x{(\log T)^2}e^{- {c_1}\log x/\log T},
\]
for large $x$. Let us now take $T$ as a function of $x$ by setting ${(\log T)^2} = \log x$ so that 
\begin{align*}
  |\psi _m^{(\beta )}(x) - \sigma_{1 - 1/\beta }^{(\beta )}(m)x| \ll x^{1+\varepsilon} \log x e^{- {(\log x)^{1/2}}} + x(\log x)e^{- {c_1}{(\log x)^{1/2}}}  \ll x^{1+\varepsilon} e^{- {c_2}{(\log x)^{1/2}}} ,  
\end{align*}
for all $\varepsilon >0$ provided that $c_2$ is a suitable constant that is less than both $1$ and $c_1$.

Next, if we assume the Riemann hypothesis, then $|{x^\rho }| = {x^{1/2}}$ and the other estimate regarding $\sum {\sigma_{1 - \gamma /\beta }^{(\beta )}(m){\gamma ^{ - 1}}}$ stays the same. Thus, the explicit formula yields
\[
|\psi _m^{(\beta )}(x) - \sigma_{1 - 1/\beta }^{(\beta )}(m)x| = O \bigg( x^{1/2} \log^2 T + \frac{x^{1+\varepsilon}\log x}{T} + \frac{{x{{\log }^2}T}}{T} \bigg)
\]
provided that $x$ is an integer. Taking $T = x^{1/2}$ leads to
\begin{align*}
\psi _m^{(\beta )}(x) & = \sigma_{1 - 1/\beta }^{(\beta )}(m)x + O({x^{1/2}}{\log ^2}{x} + x^{1/2+\varepsilon}\log x ) = \sigma_{1 - 1/\beta }^{(\beta )}(m)x + O({x^{1/2+\varepsilon}} )
\end{align*}
%%%%%%%%%%%%%%%%%%%%%%%%%%%%%%%%%%%%%%%%%%%%%%%%%%%%%%%%%%%%%%%%%%%%%%%%%%%%%%%%%%%%%%%%%%%%%%%%%%%%%%%%%%%%%%%%%%%%%%%
\section{Proof of Theorem \ref{bartz11}}
We now look at the contour integral
\[{\Upsilon ^{(\beta )}}(n,z) = \oint\nolimits_\Omega  {\frac{{\sigma_{1 - s/\beta }^{(\beta )}(n)}}{{\zeta (s)}}{e^{sz}}ds} \]
taken around the path $\Omega = [-1/2,3/2,3/2+iT_n,-1/2+iT_n]$. 

For the upper horizontal integral we have
\begin{align}
  \bigg| {\int_{-1/2 + i{T_m}}^{3/2 + i{T_m}} {\frac{{\sigma_{1 - s/\beta }^{(\beta )}(n)}}{{\zeta (s)}}{e^{sz}}ds} } \bigg| &\leqslant \int_{-1/2}^{\min(\beta,3/2)} {\bigg| {\frac{{\sigma_{1 - s/\beta }^{(\beta )}(n){e^{sz}}}}{{\zeta (\sigma  + i{T_m})}}} \bigg|d\sigma }  + \int_{\min(\beta,3/2)}^{3/2} {\bigg| {\frac{{\sigma_{1 - s/\beta }^{(\beta )}(n){e^{sz}}}}{{\zeta (\sigma  + i{T_m})}}} \bigg|d\sigma }  \nonumber \\
   &\ll T_m^{{c_1}}{n^{\beta  + 1}}{e^{ - T_m y}}\int_{-1/2}^{\min(\beta,3/2)} {{n^{ - \sigma }}{e^{\sigma x}}d\sigma }  + n{e^{ - T_m y}}\int_{\min(\beta,3/2)}^{3/2} {{e^{\sigma x}}d\sigma }  \nonumber \\
   &\to 0 \nonumber  
\end{align}
as $m \to \infty$. An application of Cauchy's residue theorem yields
\begin{align} \label{cauchy_bartz_cohen} 
\int_{ - 1/2 + i\infty }^{ - 1/2} {\frac{{\sigma_{1 - s/\beta }^{(\beta )}(n)}}{{\zeta (s)}}{e^{sz}}ds}  + \int_{ - 1/2}^{3/2} {\frac{{\sigma_{1 - s/\beta }^{(\beta )}(n)}}{{\zeta (s)}}{e^{sz}}ds}  + \int_{3/2}^{3/2 + i\infty } {\frac{{\sigma_{1 - s/\beta }^{(\beta )}(n)}}{{\zeta (s)}}{e^{sz}}ds}  = 2\pi i\varpi _n^{(\beta )}(z),
\end{align}
where for $\imag(z) > 0$ we have
\[\varpi _n^{(\beta )}(z) = \mathop {\lim }\limits_{m \to \infty } \sum_{\substack{\rho  \\ 0 < \operatorname{Im} \rho  < {T_m}}} {\frac{1}{{({k_\rho } - 1)!}}\frac{{{d^{k\rho  - 1}}}}{{d{s^{k\rho  - 1}}}}{{\bigg[ {{{(s - \rho )}^{k\rho }}\frac{{\sigma_{1 - s/\beta }^{(\beta )}(n)}}{{\zeta (s)}}{e^{sz}}} \bigg]}_{s = \rho }}} \]
with $k_\rho$ denoting the order of multiplicity of the non-trivial zero $\rho$ of the Riemann zeta-function. We denote by $\varpi _{1,n}^{(\beta )}(z)$ and by $\varpi _{2,n}^{(\beta )}(z)$ the first and second integrals on the left hand-side of \eqref{cauchy_bartz_cohen} respectively. If we operate under assumption that there are no multiple zeros, then the above can be simplified to \eqref{cohen_bartz_varpi}. This is done for the sake of simplicity, since dealing with this extra term would relax this assumption.

If $z \in \mathbb{H}$ then by \eqref{cauchy_bartz_cohen} one has
\[
2\pi i\varpi _n^{(\beta )}(z) = \varpi _{1,n}^{(\beta )}(z) + \varpi _{2,n}^{(\beta )}(z) + \varpi _{3,n}^{(\beta )}(z),
\]
where the last term is given by the vertical integral on the right of the $\Omega$ contour
\[
\varpi _{3,n}^{(\beta )}(z) = \int_{3/2}^{3/2 + i\infty } {\frac{{\sigma_{1 - s/\beta }^{(\beta )}(n)}}{{\zeta (s)}}{e^{sz}}ds}. 
\]
By the use of the Dirichlet series of $c_q^{(\beta)}(n)$ given in Lemma \ref{lemma23} and since we are in the region of absolute convergence we see that
\[
\varpi _{3,n}^{(\beta )}(z) = \sum_{q = 1}^\infty  {c_q^{(\beta )}(n)\int_{3/2}^{3/2 + i\infty } {{e^{sz - s\log q}}ds} }  =  - {e^{3z/2}}\sum_{q = 1}^\infty  {\frac{{c_q^{(\beta )}(n)}}{{{q^{3/2}}(z - \log q)}}} .
\]
By standard bounds of Stirling and the functional equation of the Riemann zeta-function we have that
\[
|\zeta(-\tfrac{1}{2} + it)| \approx (1 +|t|)
\]
%\[
%\frac{1}{{\zeta ( - \tfrac{1}{2} + it)}} \ll 1,
%\]
as $|t| \to \infty$. Therefore, we see that
\[|\varpi _{1,n}^{(\beta )}(z)| = \bigg| {\int_{ - 1/2 + i\infty }^{ - 1/2} {\frac{{\sigma_{1 - s/\beta }^{(\beta )}(n)}}{{\zeta (s)}}{e^{sz}}ds} } \bigg| = O \bigg( {n^{\beta  + 3/2}}{e^{ - x/2}}\int_0^\infty  {{e^{ - ty}}dt} \bigg) = O \bigg(\frac{{{n^{\beta  + 3/2}}{e^{ - x/2}}}}{y} \bigg),
\]
and $\varpi _{1,n}^{(\beta )}(z)$ is absolutely convergent for $y = \imag(z) > 0$. We know that $\varpi _{n}^{(\beta )}(z)$ is analytic for $y>0$ and the next step is to show that that it can be meromorphically continued for $y > - \pi$. To this end, we go back to the integral
\[\varpi _{1,n}^{(\beta )}(z) =  - \int_{ - 1/2}^{ - 1/2 + i\infty } {\frac{{\sigma_{1 - s/\beta }^{(\beta )}(n)}}{{\zeta (s)}}{e^{sz}}ds} \]
with $y>0$. The functional equation of $\zeta(s)$ yields
\begin{align} \label{aux_1}
  \varpi _{1,n}^{(\beta )}(z) &=  - \int_{ - 1/2}^{ - 1/2 + i\infty } {\sigma_{1 - s/\beta }^{(\beta )}(n)\frac{{\Gamma (s)}}{{\zeta (1 - s)}}{e^{s(z - \log 2\pi  - i\pi /2)}}ds} \nonumber \\ 
	&- \int_{ - 1/2}^{ - 1/2 + i\infty } {\sigma_{1 - s/\beta }^{(\beta )}(n)\frac{{\Gamma (s)}}{{\zeta (1 - s)}}{e^{s(z - \log 2\pi  + i\pi /2)}}ds}  \nonumber \\
   &= \varpi _{11,n}^{(\beta )}(z) + \varpi _{12,n}^{(\beta )}(z).
\end{align}
Since one has by standard bounds that
\[
\frac{{\Gamma ( - \tfrac{1}{2} + it)}}{{\zeta (\tfrac{3}{2} - it)}} \ll {e^{ - \pi t/2}} 
\]
it then follows that 
\[\varpi _{11,n}^{(\beta )}(z) \ll {n^{\beta  + 3/2}}\int_0^\infty  {{e^{ - \pi t/2}}{e^{ - \tfrac{1}{2}x - ty + t\pi /2}}dt} \ll \frac{{{e^{ - x/2}}{n^{\beta  + 3/2}}}}{y},
\]
and hence $\varpi _{11,n}^{(\beta )}(z)$ is regular for $y>0$. Similarly,
\[\varpi _{12,n}^{(\beta )}(z) \ll {n^{\beta  + 3/2}}{e^{ - x/2}}\int_0^\infty  {{e^{ - (\pi  + y)t}}dt} \ll  \frac{{{n^{\beta  + 3/2}}{e^{ - x/2}}}}{{y + \pi }} ,
\]
so that $\varpi _{12,n}^{(\beta )}(z)$ is regular for $y > -\pi$. Let us further split $\varpi _{11,n}^{(\beta )}(z)$
\[
\varpi _{11,n}^{(\beta )}(z) = \bigg( { - \int_{ - 1/2 - i\infty }^{ - 1/2 + i\infty } {}  + \int_{ - 1/2 - i\infty }^{ - 1/2} {} } \bigg){e^{s(z - \log 2\pi  - i\pi /2)}}\sigma_{1 - s/\beta }^{(\beta )}(n)\frac{{\Gamma (s)}}{{\zeta (1 - s)}}ds = I_{1,n}^{(\beta )}(z) + I_{2,n}^{(\beta )}(z).
\]
By the same technique as above, it follows that the integral $I_{2,n}^{(\beta )}(z)$ is convergent for $y < \pi$. Moreover, since $\varpi _{11,n}^{(\beta )}(z)$ is regular for $y>0$, then it must be that $I_{1,n}^{(\beta )}(z)$ is convergent for $0 < y < \pi$. Let
\[
f(n,q,s,z) = \sigma_{1 - s/\beta }^{(\beta )}(n){e^{s(z - \log 2\pi  - i\pi /2 + \log q)}}\Gamma (s).
\]
By the theorem of residues we see that
\begin{align} \label{aux_2}
   - \int_{ - 1/2 - i\infty }^{ - 1/2 + i\infty } {f(n,q,s,z)ds} &=  - \int_{1 - i\infty }^{1 + i\infty } {f(n,q,s,z)ds}  + 2\pi i\mathop {\operatorname{res} }\limits_{s = 0} f(n,q,s,z) \nonumber \\
   &=  - \int_{1 - i\infty }^{1 + i\infty } {f(n,q,s,z)ds}  + 2\pi i\sigma_1^{(\beta )}(n). 
\end{align}
This last integral integral is equal to
\begin{align} \label{aux_3}
  \int_{1 - i\infty }^{1 + i\infty } {f(n,q,s,z)ds} = 2\pi i\sum_{k = 0}^\infty  {\frac{{{{( - 1)}^k}}}{{k!}}{{\left(e^{-z} \frac{2 \pi i}{q} \right)}^k}\sigma_{1 + k/\beta }^{(\beta )}(n)}, 
\end{align}
where the last sum is absolutelty convergent. To prove this note that 
\[\operatorname{Re} ( {e^{ - (z - \log 2\pi  - i\pi /2 + \log q)}} )  = (e^{-x} 2\pi/q)\sin y > 0\]
for $0 < y < \pi$. Next, consider the path of integration with vertices $[1 \pm iT]$ and $[ -N \pm iT]$, where $N$ is an odd positive integer. By Cauchy's theorem
 \begin{align*}
  & \bigg( \int_{1 - iT }^{1 + iT } - \int_{-N + iT }^{1 + iT } - \int_{-N - iT }^{-N + iT } + \int_{-N - iT }^{1 - iT } \bigg) \left( e^{-z} \frac{2 \pi i}{q} \right)^{ - s}\sigma_{1 - s/\beta }^{(\beta )}(n)\Gamma (s)ds \\
  & = 2\pi i\sum_{k = 0}^{N-1} \frac{( - 1)^k}{k!}{\left(e^{-z} \frac{2 \pi i}{q} \right)^k}\sigma_{1 + k/\beta }^{(\beta )}(n).   
 \end{align*}
The third integral on the far left of the path can be bounded in the following way
\begin{align*}
 I_3 &:= \int_{-N - iT }^{-N + iT } \left( e^{-z} \frac{2 \pi i}{q} \right)^{ - s}\sigma_{1 - s/\beta }^{(\beta )}(n)\Gamma (s)ds \ll \int_{-T}^{T} \bigg(e^{-x} \frac{2\pi}{q} \bigg)^{N} e^{-t (y - \frac{\pi}{2})} n^{\beta + N + 1} e^{- \frac{\pi}{2} |t|} dt  \\
 &\ll \left(e^{-x} \frac{2\pi n}{q}\right)^{N} \int_{-T}^{T} e^{-t (y - \frac{\pi}{2})} e^{- \frac{\pi}{2} |t|} dt \ll  \left(e^{-x} \frac{2\pi n}{q}\right)^{N} ( e^{T(y - \pi)} + e^{- Ty} )  \\
 &\ll e^{-Nx + N\log \frac{2 \pi n}{q}}  e^{-T \min(y, \pi - y)} .
\end{align*} 
We now bound the horizontal parts. For the top one
 \begin{align*}
  I_{+} &:= \int_{-N + iT }^{1 + iT } \left( e^{-z} \frac{2 \pi i}{q} \right)^{ - s}\sigma_{1 - s/\beta }^{(\beta )}(n)\Gamma (s)ds \ll \int_{-N}^{1} \left(e^{-x} \frac{2\pi}{q}\right)^{-\sigma} e^{-T (y - \frac{\pi}{2})} n^{\beta - \sigma + 1} T^{\frac{1}{2}} e^{- T \frac{\pi}{2}} d\sigma \\
  &\ll T^{\frac{1}{2}} e^{-T y} \int_{-N}^{1} \left(e^{-x} \frac{2\pi n}{q}\right)^{-\sigma} d\sigma \ll  T^{\frac{3}{2}} e^{T (y - \pi)} \left(e^{-x} \frac{2\pi n}{q}\right)^{N} \\
	&\ll T^{\frac{1}{2}} e^{- Nx + N\log \frac{2 \pi n}{q}}  e^{-T y} ,
 \end{align*}
 and analogously for the bottom one
 \begin{align*}
  I_{-} &:= \int_{-N - iT }^{1 - iT } \left( e^{-z} \frac{2 \pi i}{q} \right)^{ - s}\sigma_{1 - s/\beta }^{(\beta )}(n)\Gamma (s)ds \ll \int_{-N}^{1} \left(e^{-x} \frac{2\pi}{q}\right)^{-\sigma} e^{T (y - \frac{\pi}{2})} n^{\beta - \sigma + 1} T^{\frac{1}{2}} e^{-T \frac{\pi}{2}} d\sigma \\
  &\ll T^{\frac{1}{2}} e^{-T (\pi - y)} \left(e^{-x} \frac{2\pi n}{q}\right)^{N} \ll T^{\frac{1}{2}} e^{-Nx + N\log \frac{2 \pi n}{q}}  e^{-T (\pi - y)} .
 \end{align*}
Let now $T = T(N)$ such that 
\[
T > \frac{N(-x + \log \frac{2 \pi n}{q})}{\min(y, \pi - y)}.
\]
It is now easy to see that all of the three parts tend to 0 as $N \to \infty$ through odd integers, and thus the result follows.
Thus, putting together \eqref{aux_3} with \eqref{aux_1} and \eqref{aux_2} gives us
\begin{align} \label{general_series}
  I_{1,n}^{(\beta )}(z) &=  - \int_{ - 1/2 - i\infty }^{ - 1/2 + i\infty } {{e^{s(z - \log 2\pi  - i\pi /2)}}\sigma_{1 - s/\beta }^{(\beta )}(n)\frac{{\Gamma (s)}}{{\zeta (1 - s)}}ds}  \nonumber \\
   &=  - \sum_{q = 1}^\infty  {\frac{{\mu (q)}}{q}\int_{ - 1/2 - i\infty }^{ - 1/2 + i\infty } {{e^{s(z - \log 2\pi  - i\pi /2)}}{q^s}\sigma_{1 - s/\beta }^{(\beta )}(n)\Gamma (s)ds} }  \nonumber \\
   &= - \sum_{q = 1}^\infty  {\frac{{\mu (q)}}{q}\bigg( { 2\pi i\sum_{k = 0}^\infty  {\frac{{{{( - 1)}^k}}}{{k!}}{{\left( {{e^{ - z}}\frac{{2\pi i}}{q}} \right)}^k}\sigma_{1 + k/\beta }^{(\beta )}(n)}  - 2\pi i\sigma_1^{(\beta )}(n)} \bigg)}  \nonumber \\
   &=  - 2\pi i\sum_{q = 1}^\infty  {\frac{{\mu (q)}}{q}\sum_{k = 0}^\infty  {\frac{{{{( - 1)}^k}}}{{k!}}{{\left( {{e^{ - z}}\frac{{2\pi i}}{q}} \right)}^k}\sigma_{1 + k/\beta }^{(\beta )}(n)} }  
\end{align}
since $\sum\nolimits_{q = 1}^\infty  {\mu (q)/q}  = 0$. Moreover,
\begin{align}
  |{(2\pi i)^{ - 1}}I_{1,n}^{(\beta )}(z)| &= \bigg| {\sum_{q = 1}^\infty  {\frac{{\mu (q)}}{q}\bigg( {\sum_{k = 0}^\infty  {\frac{{{{( - 1)}^k}}}{{k!}}{{\left({{e^{ - z}}\frac{{2\pi i}}{q}} \right)}^k}\sigma_{1 + k/\beta }^{(\beta )}(n)}  - \sigma_{1}^{(\beta )}(n)} \bigg)} } \bigg| \nonumber \\
	& = \bigg| {\sum_{q = 1}^\infty  {\frac{{\mu (q)}}{q}\bigg( {\sum_{k = 1}^\infty  {\frac{{{{( - 1)}^k}}}{{k!}}{{\left({{e^{ - z}}\frac{{2\pi i}}{q}} \right)}^k}\sigma_{1 + k/\beta }^{(\beta )}(n)}} \bigg)} } \bigg| \nonumber \\
   &\leqslant n^{\beta+1} \sum_{q = 1}^\infty  {\frac{1}{q}\bigg( {\sum_{k = 1}^\infty  {\frac{{{1}}}{{k!}}{{\left( {{e^{ - x}}\frac{{2\pi }}{q}} \right)}^k}{n^k}}} \bigg)} \nonumber \\ 
	&= n^{\beta+1} \sum_{q = 1}^\infty  {\frac{1}{q}\left( {\exp \left( {{e^{ - x}}\frac{{2\pi }}{q}n} \right)} - 1\right)}  \nonumber \\
   &\ll {n^{\beta  + 1}}{e^{2\pi n /{e^x}}}\sum_{q \leqslant [2\pi n /{e^x}]} \frac{1}{q} + \frac{{2\pi n^{\beta + 2}}}{{{e^x}}}  \sum_{q \geqslant [2\pi n/{e^x}] + 1} {\frac{1}{{{q^2}}}} \ll  {c_2}(x) , \nonumber 
\end{align}
and the series on the right hand-side of \eqref{general_series} is absolutely convergent for all $y$. Thus, this proves the analytic continuation of $\varpi _{1,n}^{(\beta )}(z)$ to $y > -\pi$. For $|y| < \pi$ one has
\begin{align*}
  \varpi _{1,n}^{(\beta )}(z) &= I_{1,n}^{(\beta)}(z) + I_{2,n}^{(\beta)}(z) + \varpi_{12,n}^{(\beta)}(z) = - 2\pi i\sum_{q = 1}^\infty  {\frac{{\mu (q)}}{q}\sum_{k = 0}^\infty  {\frac{{{{( - 1)}^k}}}{{k!}}{{\left( {{e^{ - z}}\frac{{2\pi i}}{q}} \right)}^k}\sigma_{1 + k/\beta }^{(\beta )}(n)} }   \\
  & + \int_{ - 1/2 - i\infty }^{ - 1/2} {{e^{s(z - \log 2\pi  - i\pi /2)}}\sigma_{1 - s/\beta }^{(\beta )}(n)\frac{{\Gamma (s)}}{{\zeta (1 - s)}}ds}   \\
  & - \int_{ - 1/2}^{ - 1/2 + i\infty } {{e^{s(z - \log 2\pi  + i\pi /2)}}\sigma_{1 - s/\beta }^{(\beta )}(n)\frac{{\Gamma (s)}}{{\zeta (1 - s)}}ds}   
\end{align*}
where the first term is holomorphic for all $y$, the second one for $y < \pi$ and the third for $y > -\pi$. Hence, this last equation shows the continuation of $\varpi _{n}^{(\beta )}(z)$ to the region $y > -\pi$. To complete the proof of the theorem, one then considers the function
\[\hat \varpi _n^{(\beta )}(z) = \mathop {\lim }\limits_{m \to \infty } \sum_{\substack{\rho \\ - {T_m} < \operatorname{Im} \rho  < 0}} {\frac{{\sigma_{1 - \rho /\beta }^{(\beta )}(n)}}{{\zeta '(\rho )}}{e^{\rho z}}}, \]
where the zeros are in the lower part of the critical strip and $z$ now belongs to the lower half-plane $\hat {\mathbb{H}} = \{z \in \C : \imag(z)<0 \}$. It then follows by repeating the above argument that
\[\hat \varpi _{1,n}^{(\beta )}(z) = \hat \varpi _{11,n}^{(\beta )}(z) + \hat \varpi _{12,n}^{(\beta )}(z),\]
where
\[\hat \varpi _{11,n}^{(\beta )}(z) =  - \int_{ - 1/2 - i\infty }^{ - 1/2} {{e^{s(z - \log 2\pi  - i\pi /2)}}\sigma_{1 - s/\beta }^{(\beta )}(n)\frac{{\Gamma (s)}}{{\zeta (1 - s)}}ds} \]
is absolutely convergent for $y < \pi$ and 
\[\hat \varpi _{12,n}^{(\beta )}(z) =  - \int_{ - 1/2 - i\infty }^{ - 1/2} {{e^{s(z - \log 2\pi  + i\pi /2)}}\sigma_{1 - s/\beta }^{(\beta )}(n)\frac{{\Gamma (s)}}{{\zeta (1 - s)}}ds} \]
is absolutely convergent for $y<0$. Spliting up the first integral just as before and using a similar analysis to the one we have just carried out, but using the fact that $\zeta (\bar s) = \overline {\zeta (s)}$ and choosing $T_m$ ($m \le T_m \le m+1$) such that
\[
\left| \frac{1}{{{\zeta (\sigma  - i{T_n})}}} \right| < T_n^{{c_1}},\quad  - 1 \leqslant \sigma  \leqslant 2,
\]
yields that 
\begin{align}
  \hat \varpi _{1,n}^{(\beta )}(z) &=  - 2\pi i\sum_{q = 1}^\infty  {\frac{{\mu (q)}}{q}\sum_{k = 0}^\infty  {\frac{1}{k!}{{\left( {{e^{ - z}}\frac{{2\pi i}}{q}} \right)}^k}\sigma_{1 + k/\beta }^{(\beta )}(n)} }  \nonumber \\
   &- \int_{ - 1/2 - i\infty }^{ - 1/2} {{e^{s(z - \log 2\pi  - i\pi /2)}}\sigma_{1 - s/\beta }^{(\beta )}(n)\frac{{\Gamma (s)}}{{\zeta (1 - s)}}ds}  \nonumber \\
   &+ \int_{ - 1/2}^{ - 1/2 + i\infty } {{e^{s(z - \log 2\pi  + i\pi /2)}}\sigma_{1 - s/\beta }^{(\beta )}(n)\frac{{\Gamma (s)}}{{\zeta (1 - s)}}ds}.  \nonumber 
\end{align}
Therefore, $\hat \varpi _n^{(\beta )}(z)$ admits an analytic continuation from $y<0$ to the half-plane $y < \pi$.
%%%%%%%%%%%%%%%%%%%%%%%%%%%%%%%%%%%%%%%%%%%%%%%%%%%%%%%%%%%%%%%%%%%%%%%%%%%%%%%%%%%%%%%%%%%%%%%%%%%%%%%%%%%%%%%%%%%%%%%
\section{Proof of Theorem \ref{bartz12}}
Adding up the two results of our previous section
\[\varpi _{1,n}^{(\beta )}(z) + \hat \varpi _{1,n}^{(\beta )}( z)  =  - 2\pi i\sum_{q = 1}^\infty  {\frac{{\mu (q)}}{q}\sum_{k = 0}^\infty  {\frac{1}{{k!}}\bigg\{ {{{\left( {{e^{ - z}}\frac{{2\pi i}}{q}} \right)}^k} + {{\left( { - {e^{ - z}}\frac{{2\pi i}}{q}} \right)}^k}} \bigg\}\sigma_{1 + k/\beta }^{(\beta )}(n)} }. \]
The other terms do not contribute since
\[\varpi _{2,n}^{(\beta )}(z) + \hat \varpi _{2,n}^{(\beta )}(z) = \bigg( {\int_{ - 1/2}^{3/2} {}  + \int_{3/2}^{ - 1/2} {} } \bigg){e^{sz}}\frac{{\sigma_{1 - s/\beta }^{(\beta )}(n)}}{{\zeta (s)}}ds = 0,\]
and by the Theorem \ref{bartz11} we have
\[\varpi _{3,n}^{(\beta )}(z) + \hat \varpi _{3,n}^{(\beta )}(z) = 0.\]
Consequently, we have
\begin{align} \label{pre_FE}
\varpi _n^{(\beta )}(z) + \hat \varpi _n^{(\beta )}(z) =  - \sum_{q = 1}^\infty  {\frac{{\mu (q)}}{q}\sum_{k = 0}^\infty  {\frac{1}{{k!}}\bigg\{ {{{\left( {{e^{ - z}}\frac{{2\pi i}}{q}} \right)}^k} + {{\bigg( { - {e^{ - z}}\frac{{2\pi i}}{q}} \bigg)}^k}} \bigg\}\sigma_{1 + k/\beta }^{(\beta )}(n)} }
\end{align}
for $|y| < \pi$.
Thus, once again, by the previous theorem for all $y < \pi$
\[
\varpi _n^{(\beta )}(z) =  - \hat \varpi _n^{(\beta )}(z) - \sum_{q = 1}^\infty  {\frac{{\mu (q)}}{q}\sum_{k = 0}^\infty  {\frac{1}{{k!}}\bigg\{ {{{\left( {{e^{ - z}}\frac{{2\pi i}}{q}} \right)}^k} + {{\left( { - {e^{ - z}}\frac{{2\pi i}}{q}} \right)}^k}} \bigg\}\sigma_{1 + k/\beta }^{(\beta )}(n)} } 
\]
by analytic continuation, and for $y > -\pi$
\[\hat \varpi _n^{(\beta )}(z) =  - \varpi _n^{(\beta )}(z) - \sum_{q = 1}^\infty  {\frac{{\mu (q)}}{q}\sum_{k = 0}^\infty  {\frac{1}{{k!}}\bigg\{ {{{\left( {{e^{ - z}}\frac{{2\pi i}}{q}} \right)}^k} + {{\left( { - {e^{ - z}}\frac{{2\pi i}}{q}} \right)}^k}} \bigg\}\sigma_{1 + k/\beta }^{(\beta )}(n)} }. \]
This shows that $\varpi _n^{(\beta )}(z)$ and $\hat \varpi _n^{(\beta )}(z)$ can be analytically continued over $\C$ as a meromorphic function and that \eqref{pre_FE} holds for all $z$. To prove the functional equation, we look at the zeros. If $\rho$ is a non-trivial zero of $\zeta(s)$ then so is $\bar \rho$. For $z \in \mathbb{H}$ one has
\[\varpi _n^{(\beta )}(z) = \mathop {\lim }\limits_{m \to \infty } \overline {\sum_{\substack{\rho \\ 0 < \operatorname{Im} \rho  < {T_m}}} {\overline {\frac{{\sigma_{1 - \rho /\beta }^{(\beta )}(n)}}{{\zeta '(\rho )}}{e^{\rho z}}} } } .\]
By using $\overline {\sigma_{1 - \rho /\beta }^{(\beta )}(n)} = \sigma_{1 - \bar \rho /\beta }^{(\beta )}(n)$ and since $\zeta (\bar s) = \overline {\zeta (s)}$ we get
\begin{align}
  \varpi _n^{(\beta )}(z) &= \overline {\sum_{\substack{\rho \\ 0 < \operatorname{Im} \rho  < {T_m}}} {\overline {\frac{{\sigma_{1 - \rho /\beta }^{(\beta )}(n)}}{{\zeta '(\rho )}}{e^{\rho z}}} } }  = \sum_{\substack{\rho \\ 0 < \operatorname{Im} \rho  < {T_m}}} {\frac{{\sigma_{1 - \bar \rho /\beta }^{(\beta )}(n)}}{{\zeta '(\bar \rho )}}{e^{\overline {\rho z} }}}  \nonumber \\
   &= \sum_{\substack{\rho \\ - {T_m} < \operatorname{Im} \rho  < 0}} {\frac{{\sigma_{1 - \rho /\beta }^{(\beta )}(n)}}{{\zeta '(\rho )}}{e^{\rho \bar z}}}  = \overline {\hat{\varpi} _n^{(\beta )}(\bar z)} . \nonumber 
\end{align}
Invoking \eqref{pre_FE} with $z \in \mathbb{H} $ we see that
\begin{align}
  \varpi _n^{(\beta )}(z) &= \overline {\hat \varpi _n^{(\beta )}(\bar z)}  =  - \overline {\varpi _n^{(\beta )}(\bar z)}  - \overline {\sum_{q = 1}^\infty  {\frac{{\mu (q)}}{q}\sum_{k = 0}^\infty  {\frac{1}{{k!}}\bigg\{ {{{\left( {{e^{ - \bar z}}\frac{{2\pi i}}{q}} \right)}^k} + {{\left( { - {e^{ - \bar z}}\frac{{2\pi i}}{q}} \right)}^k}} \bigg\}\sigma_{1 + k/\beta }^{(\beta )}(n)} } }  \nonumber \\
   &=  - \overline {\varpi _n^{(\beta )}(\bar z)}  - \sum_{q = 1}^\infty  {\frac{{\mu (q)}}{q}\sum_{k = 0}^\infty  {\frac{1}{{k!}}\bigg\{ {{{\left( {{e^{ - z}}\frac{{2\pi i}}{q}} \right)}^k} + {{\left( { - {e^{ - z}}\frac{{2\pi i}}{q}} \right)}^k}} \bigg\}\sigma_{1 + k/\beta }^{(\beta )}(n)} }  \nonumber 
\end{align} 
and by complex conjugation for $z \in \hat{\mathbb{H}}$, and by analytic continuation for $z$ with $y=\imag(z)=0$. This proves the functional equation \eqref{FE_cohen_bartz}. 

Another expression can be found which depends on the values of the Riemann zeta-function at odd integers
\begin{align}
  A_n^{(\beta )}(z) &=  - \sum_{q = 1}^\infty  {\frac{{\mu (q)}}{q}\bigg( {\sum_{k = 0}^\infty  {\frac{1}{{k!}}\bigg( {{{\bigg( {{e^{ - z}}\frac{{2\pi i}}{q}} \bigg)}^k} + {{\left( { - {e^{ - z}}\frac{{2\pi i}}{q}} \right)}^k}} \bigg)\sigma_{1 + k/\beta }^{(\beta )}(n)}  - 2\sigma_1^{(\beta )}(n)} \bigg)}  \nonumber \\
   &=  - \sum_{q = 1}^\infty  {\frac{{\mu (q)}}{q}\sum_{k = 1}^\infty  {\frac{1}{{k!}}\bigg( {{{\left( {{e^{ - z}}\frac{{2\pi i}}{q}} \right)}^k} + {{\bigg( { - {e^{ - z}}\frac{{2\pi i}}{q}} \bigg)}^k}} \bigg)\sigma_{1 + k/\beta }^{(\beta )}(n)} }  \nonumber \\
   &=  - \sum_{k = 1}^\infty  {\frac{1}{{k!}}{{({e^{ - z}}2\pi i)}^k}\sigma_{1 + k/\beta }^{(\beta )}(n)\sum_{q = 1}^\infty  {\frac{{\mu (q)}}{{{q^{1 + k}}}}} }  - \sum_{k = 1}^\infty  {\frac{1}{{k!}}{{( - {e^{ - z}}2\pi i)}^k}\sigma_{1 + k/\beta }^{(\beta )}(n)\sum_{q = 1}^\infty  {\frac{{\mu (q)}}{{{q^{1 + k}}}}} }  \nonumber \\
   &= - \sum_{k = 1}^\infty  {\frac{1}{{k!}}({{({e^{ - z}}2\pi i)}^k} + {{( - {e^{ - z}}2\pi i)}^k})\sigma_{1 + k/\beta }^{(\beta )}(n)\frac{1}{{\zeta (1 + k)}}}  \nonumber \\
   &=- 2\sum_{k = 1}^\infty  {\frac{{{{( - 1)}^k}{{(2\pi )}^{2k}}}}{{(2k)!}}\frac{{{e^{ - 2kz}}\sigma_{1 + 2k/\beta }^{(\beta )}(n)}}{{\zeta (2k+1)}}},  \nonumber 
\end{align}
since the even terms vanish. Finally, if $z = x + iy$, then we are left with 
\[
|A_n^{(\beta )}(z)| \le 2 n^{\beta+1} \sum_{k = 1}^{\infty} \frac{(2 \pi n e^{-x})^{2k}}{(2k)!},
\]
which converges absolutely. Thus $A_n^{(\beta )}(z)$ defines an entire function.
%%%%%%%%%%%%%%%%%%%%%%%%%%%%%%%%%%%%%%%%%%%%%%%%%%%%%%%%%%%%%%%%%%%%%%%%%%%%%%%%%%%%%%%%%%%%%%%%%%%%%%%%%%%%%%%%%%%%%%%
\section{Proof of Theorem \ref{bartz13}}
This now follows from Theorem \ref{bartz12}.
\section{Acknowledgements}
The authors acknowledge partial support from SNF grants PP00P2\_138906 as well as 200020\_149150$\backslash$1. They also wish to thank the referee for useful comments which greatly improved the quality of the paper.
%%%%%%%%%%%%%%%%%%%%%%%%%%%%%%%%%%%%%%%%%%%%%%%%%%%%%%%%%%%%%%%%%%%%%%%%%%%%%%%%%%%%%%%%%%%%%%%%%%%%%%%%%%%%%%%%%%%%%%%

\end{document}